%% file: BrLoc.tex
\documentclass[12pt]{amsart}
\usepackage{etex}
\usepackage{etoolbox}
\patchcmd{\thebibliography}{*}{}{}{}
\usepackage{amssymb}
\usepackage{cite}
\usepackage{booktabs}
\usepackage{url}
\usepackage{hyphenat}
\usepackage{mathtools}
\usepackage{pifont}
\usepackage[all,cmtip]{xy}
\usepackage{ifpdf}
\usepackage{enumitem}
\usepackage{xcolor}
\usepackage{graphicx}
\usepackage[lmargin=1in,rmargin=1in,tmargin=1in,bmargin=1in]{geometry}

\usepackage{mathrsfs}
\usepackage{xr-hyper}

\definecolor{darkblue}{rgb}{0,0,0.4} 
\usepackage[colorlinks=true, citecolor=darkblue, filecolor=darkblue, linkcolor=darkblue,urlcolor=darkblue]{hyperref} 
\usepackage[all]{hypcap}

\allowdisplaybreaks

\usepackage{tikz}
\usetikzlibrary{matrix,arrows,3d,positioning,calc}
\tikzset{zxplane/.style={canvas is zx plane at y=#1,very thin}}
\tikzset{xyplane/.style={canvas is xy plane at z=#1,very thin}}
\tikzset{yzplane/.style={canvas is yz plane at x=#1,very thin}}
\tikzstyle{intext}=[rectangle,fill=white,inner sep=1pt,outer sep=2pt, fill opacity=0.7,text opacity=1]
\tikzset{doubar/.style={double, double equal sign distance, -implies}}

 \makeatletter
 \providecommand\@dotsep{5}
 \def\listtodoname{List of Todos}
 \def\listoftodos{\@starttoc{tdo}\listtodoname}
 \makeatother

\graphicspath{{draws/}{}}

\input{defs}

\begin{document}
\title{Rank inequalities for the Heegaard Floer homology of branched covers}

\author{Kristen Hendricks}
 \address{Mathematics Department, Rutgers University\\
   New Brunswick, NJ 08854}
   \thanks{\texttt{KH was supported by NSF grant DMS-2019396 and a Sloan Research Fellowship.}}
\email{\href{mailto:kristen.hendricks@rutgers.edu}{kristen.hendricks@rutgers.edu}}

\author{Tye Lidman}
 \address{Department of Mathematics, North Carolina State University\\Raleigh, NC 27695}
   \thanks{\texttt{TL was supported by NSF grant DMS-1709702 and a Sloan Research Fellowship.}}
\email{\href{mailto:tlid@math.ncsu.edu}{tlid@math.ncsu.edu}}

\author{Robert Lipshitz}
 \address{Department of Mathematics, University of Oregon\\
   Eugene, OR 97403}
\thanks{\texttt{RL was supported by NSF grant DMS-1810893.}}
\email{\href{mailto:lipshitz@uoregon.edu}{lipshitz@uoregon.edu}}

\date{\today}

\begin{abstract}
  We show that if $L$ is a nullhomologous link in a 3-manifold $Y$ and
  $\Sigma(Y,L)$ is a double cover of $Y$ branched along $L$ then
  for each $\SpinC$-structure $\spinc$ on $Y$ there is an inequality
  \[
    \dim \HFa(\Sigma(Y,L),\pi^*\spinc;\FF_2)\geq \dim \HFa(Y,\spinc;\FF_2).
  \]
  We discuss the relationship with the $L$-space conjecture and give
  some other topological applications, as well as an analogous result for sutured Floer homology.
\end{abstract}

%


\maketitle

\tableofcontents

\renewcommand{\x}{x}
\renewcommand{\y}{y}

\section{Introduction}\label{sec:intro}
Heegaard Floer homology is a collection of invariants of
low-dimensional objects: $3$-manifolds, $4$-manifolds, knots, and so
on. Its most basic component is $\HFa$, which associates an
$\Field$-vector space $\HFa(Y,\spinc)$ to a closed, connected, oriented
$3$-manifold $Y$ together with a $\SpinC$-structure
$\spinc\in\SpinC(Y)$~\cite{OS04:HolomorphicDisks}. Our main theorem
concerns the behavior of $\HFa(Y)$ under taking branched covers:
\begin{theorem}\label{thm:main}
  Let $Y$ be a closed $3$-manifold, $L\subset Y$ an oriented
  nullhomologous link of $\ell>0$ components 
  with Seifert surface $F$, and $\spinc$ a 
  $\SpinC$-structure
  on $Y$. Let $\pi\co \Sigma(Y,L)\to Y$ be the double cover branched
  along $L$ induced by the Seifert surface $F$. Let $\pi^*\spinc$ denote
  the pullback of $\spinc$ to $\Sigma(Y,L)$
  (Definition~\ref{def:pullback-spinc}). Then, there is a spectral
  sequence with $E^1$-page given by
  \[
    \HFa(\Sigma(Y,L),\pi^*\spinc) \otimes H_*(T^{\ell - 1}) \otimes \FF_2[[\theta, \theta^{-1}]
  \]
  converging to
  \[
   \bigoplus_{\{\spinc'\mid \pi^*\spinc' = \pi^*\spinc\}} \HFa(Y,\spinc') \otimes H_*(T^{\ell - 1}) \otimes \FF_2[[\theta, \theta^{-1}].
  \]
  In particular, 
  \[
    \dim\HFa(\Sigma(Y,L),\pi^*\spinc) \geq \sum_{\{\spinc'\mid \pi^*\spinc' = \pi^*\spinc\}} \dim\HFa(Y,\spinc').
  \]
\end{theorem}
Here, $T^{\ell-1}$ denotes the $(\ell-1)$-dimensional torus, so
$H_*(T^{\ell-1})$ is isomorphic to the exterior algebra on $\ell-1$
generators. An oriented link $L\subset Y$ is nullhomologous if
$[L]=0\in H_1(Y)$; we do not require each component to be
nullhomologous. The pullback $\SpinC$-structure $\pi^*\spinc$ is
explained in Definition~\ref{def:pullback-spinc}. Throughout this
paper, Floer homology groups have coefficients in $\FF_2$ or an
$\FF_2$-module, and tensor products are over $\FF_2$ unless otherwise noted.

Theorem~\ref{thm:main} is part of a growing literature on the behavior
of Heegaard Floer homology under various kinds of covers. Previously,
Hendricks~\cite{Hendricks12:dcov-localization} used Seidel-Smith's
localization theorem for Lagrangian intersection Floer
theory~\cite{SeidelSmith10:localization} to prove a similar result for
the knot Floer homology of the double point set, as well as a spectral
sequence for the Floer homology of $2$-periodic links in
$S^3$~\cite{Hendricks:periodic-localization} (see
also~\cite{HLS:HEquivariant,Boyle:inequalities}).
Lidman-Manolescu~\cite{LidmanManolescu18:SWF-covers}
used Manolescu's homotopical refinement of
monopole Floer homology~\cite{Manolescu03:SW-spectrum-1} (see also~\cite{LidmanManolescu18:equivalence}) to prove an
analogue of Theorem~\ref{thm:main} for unbranched $p$-fold regular
covers between rational homology spheres.
Lipshitz-Treumann~\cite{LT:hoch-loc} used bordered Floer homology,
Hochschild homology, and a Yoneda-type argument to prove analogous
results for certain $2$-fold covers of $3$-manifolds with $b_1>0$ as
well as for the knot Floer homology of knots with genus $\leq
2$. (See also Remark~\ref{rem:ordinary-covs}.) Hendricks-Lipshitz-Sarkar~\cite{HLS:HEquivariant} deduced the
special case $Y=S^3$ of Theorem~\ref{thm:main} from Seidel-Smith's
localization theorem, and used it to construct concordance invariants
of knots.

Most recently, Large proved a generalization of Seidel-Smith's
localization theorem and used it to prove there are spectral sequences
for the knot Floer homology of branched double covers and $\HFa$ of
ordinary double covers under less restrictive hypotheses~\cite{Large}. We deduce
Theorem~\ref{thm:main} from Large's localization theorem. The main
work is to check that the bundle-theoretic hypotheses his result
requires hold in the setting of $\HFa$ of branched double covers (see
Section~\ref{sec:normalIso}).

Theorem~\ref{thm:main} has a number of corollaries. Recall that a
rational homology sphere $Y$ is a (modulo-2) \emph{$L$-space} if $\dim
\HFa(Y) = |H_1(Y)|$, the minimum possible dimension of $\HFa(Y)$;
this is equivalent to $\HFred(Y) = 0$.  
\begin{corollary}\label{cor:non-L}
Let $L$ be a nullhomologous link in $Y.$  If $b_1(\Sigma(Y,L)) \leq 1$ and $\HFred(\Sigma(Y,L)) = 0$, then $\HFred(Y) = 0$.  In particular, if $\Sigma(Y,L)$ is an $L$-space then $Y$ is an $L$-space.
\end{corollary}

Ni points out that when restricting to non-torsion $\SpinC$ structures, Corollary~\ref{cor:non-L} follows easily from the Thurston norm detection of Floer homology without Theorem~\ref{thm:main} and requires no constraints on $b_1$.  

Boyer-Gordon-Watson~\cite{BoyerGordonWatson13:LspaceLO} conjectured
that an irreducible rational homology sphere $Y$ is an $L$-space if
and only if $\pi_1(Y)$ does not admit a left-invariant total order.  This is known as the \emph{$L$-space conjecture}.  By
work of Boyer-Rolfsen-Wiest~\cite[Theorem 1.1]{BoyerRolfsenWiest05:orderable},
if $\pi_1(Y)$ does not admit a
left-invariant total order then neither does the fundamental group of any $3$-manifold $Y'$
which admits a non-zero degree map from $Y$.  So,
Corollary~\ref{cor:non-L} provides some further evidence for
Boyer-Gordon-Watson's conjecture. In particular, we have:  
\begin{corollary}
Let $L$ be a nullhomologous link in an irreducible rational homology sphere $Y$.   If $\Sigma(Y,L)$ is an irreducible $L$-space and satisfies the $L$-space conjecture, then so does $Y$.  
\end{corollary}

\begin{remark}
It has also been conjectured that an irreducible rational homology sphere $Y$ is an $L$-space if and only if $Y$ admits a co-orientable taut foliation.  Note that if $Y$ admits a co-orientable taut foliation and $K$ is transverse to the foliation, then $\Sigma(Y,K)$ admits a co-orientable taut foliation as well.  However, there are nullhomologous knots which cannot be transverse to the foliation (e.g. if the knot is nullhomotopic) and Theorem~\ref{thm:main} still predicts that the branched double cover should admit a co-orientable taut foliation if it is irreducible.  It would be interesting to see evidence of this through foliations.  
\end{remark}

\begin{remark}
  We do not know if the restriction that $L$ be nullhomologous in
  Theorem~\ref{thm:main} is necessary: in light of the $L$-space
  conjecture, perhaps the condition that $[L]=0\in H_1(Y;\ZZ/2\ZZ)$
  suffices. (This condition is needed to define a branched double
  cover at all.) The main step where we use that $L$ is nullhomologous
  is the proof of Lemma~\ref{lemma:works-for-cohomology}, which is
  used to prove Proposition~\ref{prop:trivialization}.
\end{remark}

Theorem~\ref{thm:main} also has some corollaries pertaining to the structure of Floer homology.
\begin{corollary}\label{cor:involution}
  If $\dim \HFa(\Sigma(Y,L),\pi^*\spinc) = \dim \HFa(Y,\spinc)$ then the
  involution $\tau_*$ on the Floer homology $\HFa(\Sigma(Y,L),\pi^*\spinc)$ of the branched double cover is the identity.
\end{corollary}

\begin{remark}
  The above corollary does not require the use of the main theorem if
  $\Sigma(Y,L)$ is an L-space or $L$ is the Borromean
  knot in $\#_{2g} S^2 \times S^1$.  We do not know any examples
  satisfying the hypothesis of the corollary when $Y$ has non-trivial
  reduced Floer homology.
\end{remark}

\begin{corollary}
Let $Y$ be a homology sphere with a non-trivial surgery to $S^3$.  Let $K$ be a knot in $Y$ such that $\Sigma(Y,K)$ is an L-space.  Then $Y = S^3$ or the Poincar\'e homology sphere.  
\end{corollary}
\begin{proof}
Let $Y$ be a homology sphere obtained by surgery on a knot in $S^3$.
By work of Ghiggini~\cite{Ghiggini08:FiberedGenusOne} and Ozsv\'ath-Szab\'o~\cite{OS04:ThurstonNorm}, either $Y$ is not an L-space, $Y$ is the Poincar\'e homology sphere, or $Y = S^3$ and $K$ is the unknot.  If $Y$ is not an L-space, Corollary~\ref{cor:non-L} implies that $\Sigma(Y,K)$ cannot be an L-space.
\end{proof}

\begin{remark}
If we additionally ask that $K$ be a knot realizing the $S^3$ surgery,
we obtain stronger constraints.  If $Y$ is $S^3$, then $K$ is
unknotted by Gordon-Luecke~\cite{GordonLuecke:knotcomplement}.  If $Y$
is the Poincar\'e homology sphere, then by Ghiggini's theorem, $K$ is
the core of surgery on the right-handed trefoil, that is, the singular
fiber of order 5 in the unique Seifert fibered structure on the
Poincar\'e homology sphere.  We can compute the double cover of
$\Sigma(2,3,5)$ branched over the singular fiber of order 5:  it is the Seifert fibered space $S^2(-1;1/3,
1/3,2/5)$.  This manifold is not an L-space (see for
example~\cite{LiscaStipsicz:OSIII}).   
Hence, if $K$ is a knot in a homology sphere $Y$ with a non-trivial surgery to $S^3$ and branched double cover an L-space
then $Y$ is $S^3$ and $K$ is the unknot.
\end{remark}

Here is another application of the main theorem:
\begin{proposition}\label{prop:tangle-rep}
Let $K$ be a knot in a prime homology sphere $Y$. Assume that $K$ has
determinant 1 and is obtained from the unknot by a rational tangle replacement.  If $\Sigma(Y,K)$ is an L-space then either $K$ is isotopic to an unknot or $\pm T_{3,5}$ in an embedded $B^3$.
\end{proposition}

We also prove an analogue of Theorem~\ref{thm:main} for sutured Floer homology:

\begin{proposition}\label{prop:sutured-inequality}
  Let $(M,\gamma)$ be a balanced sutured manifold and $L \subset M$ a
  nullhomologous link with $\ell>0$ components, and let
  $(\Sigma(M,L),\wt{\gamma})$ denote a double cover of $M$ branched over $L$ 
  with the induced sutures.  Then, there is a
  spectral sequence with $E^1$ page
  $\SFH(\Sigma(M,L), \wt{\gamma}) \otimes H_*(T^\ell) \otimes
  \FF_2[[\theta, \theta^{-1}]$ converging to
  $\SFH(M,\gamma) \otimes H_*(T^{\ell}) \otimes \FF_2[[\theta,
  \theta^{-1}]$.  In
  particular, 
\[
  \dim \SFH(\Sigma(M,L), \wt{\gamma}) \geq \dim \SFH(M, \gamma).  
\]
\end{proposition}
Note that here we have $H_*(T^\ell)$ instead of $H_*(T^{\ell - 1})$ as
in Theorem~\ref{thm:main}. Branched covers of sutured manifolds are
discussed further in Section~\ref{sec:sutured}.

There is a relationship between Theorem~\ref{thm:main} and the
Smith conjecture~\cite{Smith39:conjecture,Morgan84:SmithConjecture}. Specifically, the Smith conjecture
implies that $\ZZ/p$-actions on $S^3$ with nonempty fixed sets are standard, so $S^3$ is not the branched cover of any other
$3$-manifold. Theorem~\ref{thm:main} implies the weaker statement that
if $S^3$ is the branched cover of $Y$ then $Y$ is an
$L$-space integer homology sphere. Ozsv\'ath-Szab\'o conjecture that the only irreducible integer homology sphere
$L$-spaces are $S^3$ and the Poincar\'e homology
sphere~\cite[Section 1.5]{OS06:Clay2} (see also~\cite[Conjecture 1]{HeddenLevine:splicing});
this is sometimes referred to, somewhat drolly, as the Heegaard Floer
Poincar\'e Conjecture. Together with the Heegaard Floer Poincar\'e
Conjecture, Theorem~\ref{thm:main} implies that if $S^3$ or the
Poincar\'e homology sphere is a branched cover of $Y$ then $Y$ is
itself a connect sum of copies of the Poincar\'e sphere.

It would be interesting to obtain a similar result in Seiberg-Witten
theory, extending Lidman-Manolescu's
work~\cite{LidmanManolescu18:SWF-covers}. In particular, such a result
would perhaps entail studying Seiberg-Witten solutions on the orbifold
quotient of the branched double cover, and relating them with the
underlying manifold.  There have been a number of other results on the
Heegaard or Seiberg-Witten Floer homology of branched covers with
which it would also be interesting to
compare~\cite{Kang:transverse,Kang:strong,AlfieriKangStipsicz,KarakurtLidman:inequalities,LRS:Froyshov,LRS:Lefschetz}.
In particular, perhaps Lin-Ruberman-Saveliev's
techniques~\cite{LRS:Lefschetz} could lead to a
Seiberg-Witten-theoretic proof of Theorem~\ref{thm:main}.

This paper is organized as follows. Section~\ref{sec:background}
recalls Large's localization theorem and some background about
$K$-theory and maps of stable vector
bundles. Section~\ref{sec:normalIso} verifies the main hypothesis for
Large's localization theorem, an isomorphism between the stable relative
tangent and normal bundles to the fixed sets. Section~\ref{sec:proof}
verifies the remaining hypotheses and deduces
Theorem~\ref{thm:main}. Finally, Section~\ref{sec:applications}
discusses applications of Theorem~\ref{thm:main}, as well as Proposition~\ref{prop:sutured-inequality} 
for sutured Floer homology.

\thinspace

\noindent\textbf{Acknowledgments.}
We thank Steven Frankel, Tim Large,
Yi Ni, Rachel Roberts, and Chuck Weibel for interesting
discussions. We thank the referee for further suggestions and corrections.

\section{Background}\label{sec:background}
\subsection{Polarization data}
The following definitions are drawn from Large's paper~\cite[Section 3.2]{Large}.

\begin{definition} Let $(M, L_0, L_1)$ be a symplectic manifold and
  two Lagrangian submanifolds. A 
\emph{set of polarization data} for $(M, L_0, L_1)$ is a triple $\mathfrak p = (E, F_0, F_1)$ where
\begin{itemize}
\item $E$ is a symplectic vector bundle over $M$ 
\item $F_i$ is a Lagrangian subbundle of $E|_{L_i}$ for $i=0,1$.\end{itemize}
\end{definition}

Given $(M, L_0, L_1)$ and $\mathfrak p = (E, F_0, F_1)$ a set of polarization data for $(M, L_0, L_1)$, we may stabilize to obtain $\mathfrak p \oplus \ul{\CC} = (E\oplus \ul{\CC}, F_0\oplus \ul{\RR}, F_1 \oplus i\ul{\RR})$.

\begin{definition} Let
  $\mathfrak p = (E, F_0, F_1)$ and $\mathfrak p'= (E', F'_0, F'_1)$ be two sets of polarization data for $(M, L_0, L_1)$. An \emph{isomorphism of polarization data} is an isomorphism of symplectic vector bundles
\[\alpha \co E \rightarrow E'\]
\noindent such that there are homotopies of Lagrangian subbundles of
$E'|_{L_i}$ between $\alpha(F_i)$ and $F_i'$ for $i=0,1$ (so that the
subbundles stay Lagrangian throughout the homotopy). A \emph{stable isomorphism of polarization data} between $\mathfrak p$ and $\mathfrak p'$ is an isomorphism of polarization data between $\mathfrak p \oplus \ul{\CC}^n$ and $\mathfrak p' \oplus \ul{\CC}^{n'}$ for some $n,n'$.\end{definition}

One special case of this definition will be of particular importance.
Suppose $(M,L_0,L_1)$ is
equipped with a symplectic involution preserving $L_0$ and $L_1$
setwise. Let $(M^\fix, L_0^\fix, L_1^\fix)$ denote the fixed sets
under the involution.
Then there are two sets of polarization data for
$(M^\fix, L_0^\fix, L_1^\fix)$: the \emph{tangent polarization}
$(TM^\fix, TL_0^\fix, TL_1^\fix)$ consisting of the tangent bundles to
$M^\fix$ and $L_i^\fix$, and the \emph{normal polarization}
$(NM^\fix, NL_0^\fix, NL_1^\fix)$ consisting of the normal bundles to $M^{\fix} \subset M$ and $L_i^{\fix}\subset L_i$.

\begin{definition}
  With notation as above, a \emph{stable tangent-normal isomorphism} is a stable
  isomorphism of polarization data between the tangent polarization
  $(TM^\fix, TL_0^\fix, TL_1^\fix)$ and the normal polarization
  $(NM^\fix, NL_0^\fix, NL_1^\fix)$.
\end{definition}

\subsection{Large's localization theorem}\label{sec:Large}

The following is an immediate consequence of Large's construction of
equivariant Floer homology and its formal properties (including his
localization isomorphism):
\begin{theorem}\label{thm:Large}\cite{Large} Suppose that
  \begin{enumerate}[label=(L\arabic*)]
  \item\label{item:symp-hyp} $M$ is an exact symplectic manifold and convex at infinity, and $L_0$, $L_1$ are exact Lagrangians such that either $L_0$ and $L_1$ are compact or $M$ is a symplectization near infinity and $L_0$ and $L_1$ are conical and disjoint near infinity;
  \item\label{item:tau-hyp} $\tau$ is a symplectic involution of $M$
    preserving the $L_i$ setwise, and $(M^\fix, L^\fix_0, L^\fix_1)$ are the
    fixed sets under $\tau$; and
  \item there is a stable tangent-normal isomorphism between the
    data $(NM^\fix,
    NL_0^\fix, NL_1^\fix)$ and $(TM^\fix, TL_0^\fix, TL_1^\fix)$.
  \end{enumerate}
  Then there is an ungraded spectral sequence with $E_1$-page
  isomorphic to $\HF(L_0,L_1)\otimes_{\FF_2}\FF_2[[\theta,\theta^{-1}]$
  converging to $\HF(L_0^\fix,L_1^\fix)\otimes_{\FF_2}\FF_2[[\theta,\theta^{-1}]$.  In
  particular, there is a rank inequality
  $\dim_{\FF_2}\HF(L_0,L_1)\geq
  \dim_{\FF_2}\HF(L_0^\fix,L_1^\fix)$.
\end{theorem}

\begin{proof}
  This argument is essentially given by Large~\cite[Proof of Theorem 1.4]{Large}; we
  summarize it here. First, under the hypotheses~\ref{item:symp-hyp}
  and~\ref{item:tau-hyp}, Seidel-Smith~\cite[Section 3.2]{SeidelSmith10:localization} couple the $\dbar$-equation on
  $(M,L_0,L_1)$ to Morse theory on $\RR P^\infty$ to
  construct $\ZZ/2\ZZ$-equivariant Floer homology groups
  $HF^{SS}_{\ZZ/2\ZZ}(L_0, L_1)$ and a spectral sequence 
  \begin{equation}\label{eq:ssss}
    \HF(L_0, L_1)\otimes_{\FF_2}\FF_2[[\theta]]\Rightarrow \HF^{SS}_{\ZZ/2\ZZ}(L_0, L_1).
  \end{equation}
  (See also~\cite{HLS:HEquivariant} for an equivalent construction.)
  Under the same hypotheses, Large uses a blow-up construction
  analogous to Kronheimer-Mrowka's construction of monopole Floer
  homology to define another equivariant cohomology group
  $\HF^{KM}_{\ZZ/2\ZZ}(L_0,L_1)$. He then shows~\cite[Theorem 1.2 or
  Theorem 8.1]{Large} that 
  \begin{equation}
    \HF^{KM}_{\ZZ/2\ZZ}(L_0,L_1)\otimes_{\FF_2[\theta]}\FF_2[[\theta]]\cong \HF^{SS}_{\ZZ/2\ZZ}(L_0, L_1).
  \end{equation}

  Given a set of polarization data $\mathfrak{p}$, under
  hypothesis~\ref{item:symp-hyp} Large also constructs a
  Floer homology twisted by $\mathfrak{p}$, $\twHF(L_0, L_1;
  \mathfrak{p})$. In the special case that $\mathfrak{p}_N$ is the
  normal polarization for $(M^\fix,L_0^\fix,L_1^\fix)$, he shows~\cite[Theorem 1.1]{Large} that there is an isomorphism
  \begin{equation}\label{eq:blow-up}
    HF^{KM}_{\ZZ/2\ZZ}(L_0,L_1)\otimes_{\FF_2[\theta]}\FF_2[\theta,\theta^{-1}] \cong \twHF(L_0^\fix, L_1^\fix; \mathfrak{p}_N).
  \end{equation}
  On the other hand, using what he calls the total Steenrod square
  (coming from the $\ZZ/2\ZZ$-action on $M\times M$ exchanging the
  factors), he shows~\cite[Proposition 9.5]{Large} that for the tangent polarization $\mathfrak{p}_T$,
  \begin{equation}
    \twHF(L_0^\fix, L_1^\fix; \mathfrak{p}_T) \cong HF(L_0^\fix, L_1^\fix) \otimes_{\FF_2} \FF[\theta,\theta^{-1}].
  \end{equation}
 (This uses the action filtration. In
  particular, exactness of the Lagrangians is used here.)
  The existence of a tangent-normal isomorphism yields an isomorphism
  \begin{equation}
    \twHF(L_0^\fix, L_1^\fix; \mathfrak{p}_N) \cong \twHF(L_0^\fix, L_1^\fix; \mathfrak{p}_T).
  \end{equation}
  Combining these formulas gives the spectral sequence. Finally, the
  rank inequality over $\FF_2$ follows from the universal coefficient
  theorem, bearing in mind that $\FF_2$ and
  $\FF_2[[\theta,\theta^{-1}]$ are fields.
\end{proof}

We note a minor refinement of Large's result. Let $P(L_0,L_1)$ denote
the space of paths from $L_0$ to $L_1$. For $x\in L_0\cap L_1$ there
is a corresponding constant path $[x]\in P(L_0,L_1)$.
Two points $x,y\in L_0\cap L_1$ can be connected by a Whitney disk if and only if $[x]$
and $[y]$ lie in the same component of $P(L_0,L_1)$. So, the Floer complex
$\CF(L_0,L_1)$ decomposes as a direct sum
\begin{equation}\label{eq:decomp-1}
  \CF(L_0,L_1)=\bigoplus_{\spinc\in\pi_0 P(L_0,L_1)}\CF(L_0,L_1;\spinc).
\end{equation}
The relevance for us is that, in Heegaard Floer homology, the path
components of $P(\TT_\alpha,\TT_\beta)$ correspond to the
$\SpinC$-structures on $Y$.

In the setting of Theorem~\ref{thm:Large}, there is an inclusion map
$\iota\co P(L_0^\fix,L_1^\fix)\into P(L_0,L_1)$, inducing a set map
$\iota_*\co \pi_0P(L_0^\fix,L_1^\fix)\to \pi_0P(L_0,L_1)$. The map
$\iota_*$ is typically neither injective nor surjective.
Large's invariant $\twHF(L_0^\fix,L_1^\fix;\mathfrak{p})$
decomposes along $\pi_0P(L_0^\fix,L_1^\fix)$ as
\begin{align}\label{eq:decomp-2}
  \twHF(L_0^\fix,L_1^\fix;\mathfrak{p})&=\bigoplus_{\spinc\in\pi_0P(L_0^\fix,L_1^\fix)}\twHF(L_0^\fix,L_1^\fix;\mathfrak{p};\spinc)
\end{align}
and hence also as
\begin{align}\label{eq:decomp-2b}
  \twHF(L_0^\fix,L_1^\fix;\mathfrak{p})&=\bigoplus_{\wt{\spinc}\in\pi_0P(L_0,L_1)}\bigoplus_{\spinc\in\iota_*^{-1}(\wt{\spinc})}\twHF(L_0^\fix,L_1^\fix;\mathfrak{p};\spinc).
\end{align}
Both the invariants $\HF_{\ZZ/2\ZZ}^{KM}(L_0,L_1)$ and
$\HF_{\ZZ/2\ZZ}^{SS}(L_0,L_1)$ and the Seidel-Smith spectral
sequence~\eqref{eq:ssss} decompose along $\tau$-orbits in $\pi_0P(L_0,L_1)$
as 
\begin{align}
  \HF_{\ZZ/2\ZZ}^{SS/KM}(L_0,L_1)
  =\bigoplus_{[\wt{\spinc}]\in\pi_0P(L_0,L_1)/\tau}&\HF_{\ZZ/2\ZZ}^{SS/KM}(L_0,L_1;[\wt{\spinc}])\label{eq:decomp-3}\\
  \bigoplus_{\wt{\spinc}\in [\wt{\spinc}]}\HF(L_0, L_1;\wt{\spinc})\otimes_{\FF_2}\FF_2[[\theta]]\Rightarrow &\HF^{SS}_{\ZZ/2\ZZ}(L_0, L_1;[\wt{\spinc}]).\label{eq:decomp-4}
\end{align}
Further, the equivariant Steenrod square, which comes from Floer
theory on $M^\fix\times M^\fix$, respects the
decompositions~\eqref{eq:decomp-1} and~\eqref{eq:decomp-2}, and the localization isomorphism~\eqref{eq:blow-up} respects the
decompositions~\eqref{eq:decomp-2b} and~\eqref{eq:decomp-3}. (If $\wt{\spinc}$ is not fixed by $\tau$ then
$\HF_{\ZZ/2\ZZ}^{SS/KM}(L_0,L_1;[\wt{\spinc}])\cong
\HF(L_0,L_1;\wt{\spinc})$ for either representative $\wt{\spinc}$ of
$[\wt{\spinc}]$ and, in particular, is $\theta$-torsion.) 

So, we have:
\begin{proposition}\label{prop:Large-decomp}
  Under the same hypotheses as Theorem~\ref{thm:Large}, for each
  $\wt{\spinc}\in\pi_0P(L_0,L_1)$ there is a spectral sequence
  \[
    \HF(L_0,L_1;\wt{\spinc})\otimes_{\FF_2}\FF_2[[\theta,\theta^{-1}]\Rightarrow \bigoplus_{\spinc\in\iota_*^{-1}(\wt{\spinc})}\HF(L_0^\fix,L_1^\fix;\spinc)\otimes_{\FF_2}\FF_2[[\theta,\theta^{-1}]
  \]
  and a rank inequality
  \[
    \dim_{\FF_2}\HF(L_0,L_1;\wt{\spinc})\geq
    \sum_{\spinc\in\iota_*^{-1}(\wt{\spinc})}\dim_{\FF_2}\HF(L_0^\fix,L_1^\fix;\spinc).
  \]
\end{proposition}

\subsection{\texorpdfstring{$K$}{K}-theory and maps of stable vector bundles}\label{sec:K1}
In this section we recall some notions related to the $K$-theory of
complex vector bundles.  We consider bundles over a CW complex $X$
which is homotopy equivalent to a finite CW complex.

We focus particularly on maps between stable bundles. The main goal is
to recall that the set of homotopy classes of isomorphisms between
stable bundles is an affine copy of $K^1(X)$ and hence, under
favorable conditions, there is a Chern character isomorphism from this
set to the odd cohomology of $X$.

\begin{definition}
  Let $E,E'$ be complex vector bundles over a base $X$. A \emph{stable
    isomorphism} from $E$ to $E'$ is a bundle isomorphism
  \[
    f\co E\oplus \ul{\CC}^N\to E'\oplus\ul{\CC}^N
  \]
  for some integer $N$.  Stable isomorphisms compose in the obvious
  way.

  Two stable isomorphisms
  $f_i\co E\oplus \ul{\CC}^{N_i}\to E'\oplus \ul{\CC}^{N_i}$, $i=1,2$, are
  \emph{homotopic} if there is an integer $M\geq \max\{N_1,N_2\}$ and
  a homotopy between
  \[
    f_1\oplus\Id_{\CC^{M-N_1}},f_2\oplus\Id_{\CC^{M-N_2}}\co
    E\oplus\ul{\CC}^M\to E'\oplus\ul{\CC}^M.
  \]

  Let $\Iso(E,E')$ denote the set of homotopy classes of stable
  isomorphisms from $E$ to $E'$.
\end{definition}

\begin{definition}\label{def:action}
  Let $\ul{\CC}^0$ denote the trivial $0$-dimensional vector
  bundle over $X$. Let $E,E'$ be vector bundles over $X$ so that
  $\Iso(E,E')\neq\emptyset$. Given $[f]\in\Iso(E,E')$ and
  $[g]\in\Iso(\ul{\CC}^0,\ul{\CC}^0)$ define
  $[f*g]\in\Iso(E,E')$ as follows. The map $f$ is a bundle
  isomorphism $E\oplus\ul{\CC}^N\to E'\oplus\ul{\CC}^N$ and the map $g$ is a
  bundle isomorphism $\ul{\CC}^M\to\ul{\CC}^M$, for some integers $M$,
  $N$. Then $[f*g]$ is the homotopy class of the bundle isomorphism
  $f\oplus g\co E\oplus\ul{\CC}^N\oplus\ul{\CC}^M\to E'\oplus\ul{\CC}^N\oplus\ul{\CC}^M$.
\end{definition}

\begin{proposition}\label{prop:torsor}
  Let $X$ be a CW complex homotopy equivalent to a finite CW
  complex. Then, given complex vector bundles $E,E'$ over $X$ with $\Iso(E,E')\neq\emptyset$, 
  Definition~\ref{def:action} defines an action of
  $\Iso(\ul{\CC}^0,\ul{\CC}^0)$ on
  $\Iso(E,E')$. Further, this action makes $\Iso(E,E')$ into a torsor
  over $\Iso(\ul{\CC}^0,\ul{\CC}^0)$.
\end{proposition}
\begin{proof}
  The key point is that given a bundle $E''$ and bundle isomorphisms
  $k,\ell\co E''\to E''$ the bundle isomorphisms
  \[
    k\oplus\ell,\ (k\circ\ell)\oplus\Id\co E''\oplus E''\to E''\oplus E''
  \]
  are homotopic. To see this, note that given an invertible $2\times 2$ matrix
  $A$ over $\CC$ there is an induced automorphism $A\co E''\oplus E''\to
  E''\oplus E''$. The homotopy between $k\oplus\ell$ and
  $(k\circ\ell)\oplus\Id$ is given by
  \[
    \begin{pmatrix}
      k & 0\\
      0 & \Id
    \end{pmatrix}
    \begin{pmatrix}
      \cos(\pi t) & -\sin(\pi t)\\
      \sin(\pi t) & \cos(\pi t)
    \end{pmatrix}
    \begin{pmatrix}
      \Id & 0\\
      0 & \ell
    \end{pmatrix}
    \begin{pmatrix}
      \cos(\pi t) & \sin(\pi t)\\
      -\sin(\pi t) & \cos(\pi t)
    \end{pmatrix}
  \]
  (cf.~\cite{Atiyah89:K-book}).

  Using this observation, if $f'=f\oplus \Id_{\CC^K}$ then $f'\oplus
  g\sim f\oplus g\oplus \Id_{\CC^K}$. (Here, the bundle $E''$ in the
  key observation is a trivial bundle.) It follows easily that $[f]*[g]$
  is independent of the choices of representatives $f$ and
  $g$. Next, for appropriate choices of representatives, $[f]*([g]*[h])$ and
  $([f]*[g])*[h]$ agree on the nose.  
  It remains to see that for any pair of elements
  $f,h\in \Iso(E,E')$ there is a $g\in\Iso(\ul{\CC}^0,\ul{\CC}^0)$ so
  that $f*g=h$. 

  Given $[f]\in\Iso(E,E')$, composition with $f$ gives a bijection
  between the $\Iso(\ul{\CC}^0,\ul{\CC}^0)$-sets $\Iso(E,E)$ and
  $\Iso(E,E')$. So, it suffices to prove freeness and transitivity of
  the action in the case that $[f],[h]\in\Iso(E,E)$.
  
  We start with transitivity of the action.
  To keep notation simple, replace $E$ by its sum with a
  high-dimensional trivial bundle, so $f,h\co E\to E$.
  Choose a bundle $F$ so that $E\oplus F$ is isomorphic to
  a trivial bundle $\ul{\CC}^N$. Let $\phi\co E\oplus
  F\stackrel{\cong}{\lra}\ul{\CC}^N$ be an isomorphism. Then we have isomorphisms
  \[
    \phi\circ (f\oplus\Id_F)\circ \phi^{-1},\ \phi\circ (h\oplus \Id_F)\circ \phi^{-1}\co \ul{\CC}^N\to\ul{\CC}^N.
  \]
  Let
  \[
    g= \phi\circ (f\oplus\Id_F)^{-1}\circ
    (h\oplus\Id_F)\circ\phi^{-1}\co \ul{\CC}^N\to\ul{\CC}^N.
  \]
  We claim that $f*g\sim h$. Indeed, 
  applying the key point above, we have
  \begin{align*}
    f*g&=(\Id_E\oplus \phi)\circ(f\oplus f^{-1}\oplus\Id_F)\circ
         (\Id_E\oplus h\oplus \Id_F)\circ(\Id_E\oplus \phi^{-1})\\
       &\sim(\Id_E\oplus \phi)\circ((f\circ f^{-1}\circ h)\oplus \Id_E\oplus
         \Id_F)\circ(\Id_E\oplus \phi^{-1})\\
       &=(\Id_E\oplus\phi)\circ(h\oplus\Id_E \oplus \Id_F)\circ
         (\Id_E\oplus\phi^{-1})\\
       &=h\oplus\Id_{\CC^N},
  \end{align*}
  as desired.

  Similarly, for freeness, suppose that $[f]*[g]=[f]*[g']$. By
  stabilizing as needed, we may assume that $f*g$ and $f*g'$ are
  homotopic maps $E\oplus\ul{\CC}^M\to E\oplus\ul{\CC}^M$. Let $F$ be
  as above. Then
  \[
    f\oplus \Id_F\oplus g\sim f\oplus \Id_F\oplus g'\co
    E\oplus F\oplus \ul{\CC}^M\to E\oplus F\oplus \ul{\CC}^M,
  \]
  so 
  \[
    [\phi\circ (f\oplus \Id_F)]\oplus g\sim [\phi\circ (f\oplus \Id_F)]\oplus g'\co
    \ul{\CC}^N\oplus \ul{\CC}^M\to \ul{\CC}^N\oplus \ul{\CC}^M.
  \]
  Thus, composing both sides with
  $(\phi\circ (f\oplus\Id_F))^{-1}\oplus\Id_{\CC^M}$, the maps
  $\Id_{\CC}^N\oplus g$ and $\Id_{\CC}^N\oplus g'$ are homotopic, so
  $[g]=[g']$. This completes the proof.
\end{proof}

\begin{remark}
  Here is an alternative understanding of
  Proposition~\ref{prop:torsor}. The stable automorphisms of the
  trivial bundle over $X$ are the same as $\pi_1(\Map(X,BU))$, based
  at the constant map. The group of stable automorphisms of a nontrivial bundle
  is the fundamental group of a different path component of
  $\Map(X,BU)$. Since $BU$ is an $h$-space, all path components of
  $\Map(X,BU)$ have isomorphic fundamental groups.
\end{remark}

We can extend the Chern character to stable isomorphisms. Recall that
given an automorphism $f$ of the trivial bundle $\ul{\CC}^N$ over $X$,
the mapping cylinder $\Cyl(f)$ of $f$ is a bundle over $X\times[0,1]$
equipped with a trivialization of
$\Cyl(f)|_{X\times\{0,1\}}$. Specifically,
$\Cyl(f)=\bigl((\ul{\CC}^N\times[0,1])\amalg \ul{\CC}^N\bigr)/\sim$ where
$(x,v,1)\in \ul{\CC}^N\times\{1\}$ is identified to $(x,f(x)(v))\in
\ul{\CC}^N$, and the trivializations over $X\times\{0\}$ and $X$ are
the standard ones. Equivalently, $\Cyl(f)$ is the trivial bundle over
$X\times[0,1]$ where the trivializations over $X\times\{0\}$ and
$X\times\{1\}$ are the standard trivialization and $f$, respectively.

A (stable) trivialization of the relative bundle
$(\Cyl(f),\Cyl(f)|_{X\times\{0,1\}})$ is equivalent to a (stable)
homotopy between $f$ and the identity map. Consequently, the Chern character of
$\Cyl(f)$ is an element
\[
  \chern(f)\in H^{\mathrm{even}}(X\times[0,1],X\times\{0,1\};\QQ)=H^{\mathrm{even}}(SX;\QQ)=H^{\mathrm{odd}}(X;\QQ)
\] 
and the map
\[
  \chern\co \Iso(\ul{\CC}^0,\ul{\CC}^0)\otimes\QQ\to H^{\mathrm{odd}}(X;\QQ)
\]
is an isomorphism.

By Proposition~\ref{prop:torsor}, given an element $[f]\in\Iso(E,E')$,
any other element $[h]\in \Iso(E,E')$ can be written as $[h]=[f]*[g]$ for a
unique $[g]\in\Iso(\ul{\CC}^0,\ul{\CC}^0)$. Define
\[
  \chern_f([h])=\chern([g])\in H^{\mathrm{odd}}(X;\QQ).
\]

In particular, in the case $E=E'$ we can take $f=\Id$, and we have a
canonical choice of Chern character $\chern\co \Iso(E,E)\to
H^{\mathrm{odd}}(X;\QQ)$. Here is an alternative description of the
Chern character in this case. Given $h\in\Iso(E,E)$ the mapping torus
$T_h$ of $h$ is a vector bundle over $X\times S^1$. The maps $X\into X\times
S^1\onto X$ and the canonical generator $[S^1]\in H^1(S^1)$ identify $H^{\mathrm{even}}(X\times S^1)\cong
H^{\mathrm{even}}(X)\oplus H^{\mathrm{odd}}(X)$; the map
$H^{\mathrm{odd}}(X)\to H^{\mathrm{even}}(X\times S^1)$ is $a\mapsto
a\times [S^1]$. We have:
\begin{lemma}
  For $h\in\Iso(E,E)$, the Chern character $\chern(h)$ is the image of
  the Chern character of $T_h$ in $H^{\mathrm{odd}}(X)$.
\end{lemma}
\begin{proof}
  We first reduce to the case that $E$ is the trivial bundle. Write
  $h=\Id*g$, where $g\in\Iso(\ul{\CC}^0,\ul{\CC}^0)$. On the one hand,
  $\chern(h)=\chern(g)$. On the other hand, $T_h$ is stably isomorphic
  to $T_\Id\oplus T_g$, so $\chern(T_\Id\oplus T_g)=\chern(T_\Id)+\chern(T_g)$.
  Since $T_\Id=E\times S^1$,
  $\chern(T_\Id)\in H^{\mathrm{even}}(X)\subset H^{\mathrm{even}}(X\times S^1)$.
  Hence, the image of $\chern(T_\Id)$ in $H^{\mathrm{odd}}(X)$ vanishes, so the image of $\chern(T_\Id\oplus T_g)$ is the same as the image of $\chern(T_g)$. 

  So, it remains to show that for $g\in\Iso(\ul{\CC}^0,\ul{\CC}^0)$,
  the class $\chern(g)$ agrees with the image of $\chern(T_g)$.
  Fix a distinguished point $1\in S^1$.
  There is a commutative diagram of bundles and trivializations
  \[
    \xymatrix{
      (T_g,\emptyset)\ar[d]\ar[r] & (T_g,T_g|_{X\times \{1\}}) \ar[d] & (\Cyl(g),\Cyl(g)|_{X\times\{0,1\}})\ar[l]\ar[d]\\
      (X\times S^1,\emptyset)\ar[r] & (X\times S^1,X\times\{1\}) & (X\times[0,1],X\times\{0,1\}).\ar[l]
    }
  \]
  (In the top row, the entries $T_g|_{X\times \{1\}}$ and
  $\Cyl(g)|_{X\times\{0,1\}}$ are shorthand for the fixed
  trivializations of these bundles.)  Further, naturality of the
  cohomology cross product, the definition of the fundamental class in
  cohomology, and naturality of the Chern character give a
  commutative diagram
  \[
    \begin{tikzpicture}[xscale=5.9, yscale=2]
      \node at (0,.5) (chl) {$\chern(T_g)$};
      \node at (1,.5) (chc) {$\chern(T_g,T_g|_{X\times\{1\}})$};
      \node at (2,.5) (chr) {$\chern(\Cyl(g),\Cyl(g)|_{X\times\{0,1\}})$};
      \node[rotate=270] at (0,.25) (inl) {$\in$};
      \node[rotate=270] at (1,.25) (inl) {$\in$};
      \node[rotate=270] at (2,.25) (inl) {$\in$};
      \node at (0,0) (tl) {$H^{\mathrm{even}}(X\times S^1;\QQ)$};
      \node at (1,0) (tc) {$H^{\mathrm{even}}(X\times S^1,X\times 1;\QQ)$};
      \node at (2,0) (tr) {$H^{\mathrm{even}}(X\times[0,1],X\times\{0,1\};\QQ)$};
      \node at (1,-1) (bc) {$H^{\mathrm{odd}}(X;\QQ)$};
      \draw[->] (bc) to node[below]{\lab{\times[S^1]}} (tl);
      \draw[->] (bc) to node[right]{\lab{\times[S^1]}} node[left]{\lab{\cong}} (tc);
      \draw[->] (bc) to node[below]{\lab{\times[0,1]}} node[above]{\lab{\cong}} (tr);
      \draw[->] (tc) to (tl);
      \draw[->] (tc) to node[above]{\lab{\cong}} (tr);
      \draw[|->] (chc) to (chl);
      \draw[|->] (chc) to (chr);
    \end{tikzpicture}
  \]
  The Chern character of $g$ is the preimage of
  $\chern(\Cyl(g),\Cyl(g)|_{0,1})$ under the right diagonal
  isomorphism. Thus, the left diagonal arrow sends the Chern character
  of $g$ to the Chern character of $T_g$, as claimed.
\end{proof}

This Chern character map is natural in the following sense:
\begin{lemma}
  Let $G\co X\to Y$ be a continuous map, $E,E'$ be complex vector bundles over
  $Y$, and $[f],[h]\in\Iso(E,E')$. There are induced isomorphisms
  $[G^*f], [G^*h]\in\Iso(G^*E,G^*E')$. Then,
  \[
    \chern_{G^*f}([G^*h])=G^*\chern_f([h]).
  \]
  In particular, if $E=E'$ then 
  \[
    \chern([G^*h])=G^*\chern([h]).
  \]
\end{lemma}
\begin{proof}
  This is immediate from the definitions.
\end{proof}

Finally, the Chern character respects composition:
\begin{lemma}
  If $[h_1],[h_2]\in\Iso(E,E)$ then
  \[
    \chern([h_2\circ h_1])=\chern(h_1)+\chern(h_2).
  \]

  More generally, given bundles $E_1,E_2,E_3$ and maps
  $[f_1],[h_1]\in\Iso(E_1,E_2)$ and $[f_2],[h_2]\in\Iso(E_2,E_3)$ we
  have
  \[
    \chern_{f_2\circ f_1}([h_2\circ h_1])=\chern_{f_1}([h_1])+\chern_{f_2}([h_2]).
  \]
\end{lemma}
\begin{proof}
  We prove the more general statement; the special case follows by
  taking $f_1=f_2=\Id$.
  Write $[h_1]=[f_1*g_1]$ and $[h_2]=[f_2*g_2]$. As in the beginning
  of the proof of Proposition~\ref{prop:torsor},
  $[h_2\circ h_1]=[(f_2\circ f_1)*g_1*g_2]$. Hence
  \[
    \chern_{f_2\circ f_1}([h_2\circ h_1])=\chern(g_1*g_2).
  \]
  It is immediate from the construction of the Chern character for
  maps of trivial bundles and additivity of the usual Chern
  character for complex vector bundles that for
  $g_1,g_2\in\Iso(\ul{\CC}^0,\ul{\CC}^0)$, $\chern(g_1*g_2)=\chern(g_1)+\chern(g_2)$.
  The result follows.
\end{proof}

\begin{remark}\label{rem:cx-is-symp}
  Since the inclusion of the unitary group into the symplectic group
  is a homotopy equivalence, the $K$-theory of complex vector bundles
  is the same as the $K$-theory of symplectic vector bundles. In
  particular, one can take the Chern character of symplectic vector
  bundles and isomorphisms between them, and the results of this
  section hold in the symplectic case as well.
\end{remark}

\section{The stable tangent-normal isomorphism}\label{sec:normalIso}
Let $\HD=(\Sigma_g,\alphas,\betas,z,w)$ be a doubly-pointed Heegaard
diagram for a nullhomologous knot $K$ in a $3$-manifold $Y$,
$\pi\co \Sigma(Y,K)\to Y$ a double cover of $Y$ branched along $K$,
and $\wt{K}=\pi^{-1}(Y)$. There is an induced doubly-pointed Heegaard
diagram
$\wt{\HD}=(\wt{\Sigma}_{2g},\wt{\alphas},\wt{\betas},\wt{z},\wt{w})$
for $(\Sigma(Y,K),\wt{K})$ as follows. Viewing $\Sigma$ as a subset of
$Y$, $\wt{\Sigma}=\pi^{-1}(\Sigma)$. The preimage of $\alphas$
(respectively $\betas$) is a collection of $2g$ circles $\wt{\alphas}$
(respectively $\wt{\betas}$) in $\wt{\Sigma}$, and the preimage of $z$
(respectively $w$) is a point $\wt{z}$ (respectively $\wt{w}$) in
$\wt{\Sigma}$.

The covering involution $\tau\co \Sigma(Y,K)\to\Sigma(Y,K)$ induces an
involution $\tau$ of $\wt{\HD}$. A complex structure on
$\Sigma$ induces a $\tau$-equivariant complex structure on
$\wt{\Sigma}$, which makes $\Sym^{2g}(\wt{\Sigma})$ into a smooth
complex manifold. The involution $\tau$ induces a smooth involution of
$\Sym^{2g}(\wt{\Sigma})$, by
\[
  \tau(\{x_1,\dots,x_{2g}\})=\{\tau(x_1),\dots,\tau(x_{2g})\}.
\]

The goal of this section is to prove:
\begin{proposition}\label{prop:trivialization} 
  Let $\HD=(\Sigma_g,\alphas,\betas,z,w)$ be a doubly-pointed Heegaard
  diagram for a nullhomologous knot $K$ in a closed 3-manifold
  $Y$ and let $\wt{\HD}=(\wt{\Sigma},\wt{\alphas},\wt{\betas},\wt{z},\wt{w})$
  be the branched double cover of $\HD$, which is a doubly-pointed
  Heegaard diagram for $(\Sigma(Y,K),\wt{K})$. Then there is a stable
  tangent-normal isomorphism
  \[
    \bigl(T\Sym^{2g}(\wt{\Sigma}\setminus\{\wt{z}\})^\fix,T\TT_{\wt{\alpha}}^\fix,T\TT_{\wt{\beta}}^\fix\bigr)\cong
    \bigl(N\Sym^{2g}(\wt{\Sigma}\setminus\{\wt{z}\})^\fix,N\TT_{\wt{\alpha}}^\fix,N\TT_{\wt{\beta}}^\fix\bigr).
  \]
\end{proposition}

We start by noting that the fixed set of the involution is familiar:
\begin{lemma}\label{lem:fixed-set}
  There is a $\tau$-equivariant K\"ahler form on
  $\Sym^{2g}(\wt{\Sigma}\setminus\{\wt{z}\})$ and a K\"ahler form on
  $\Sym^g(\Sigma\setminus\{z\})$, so that the fixed set
  $\Sym^{2g}(\wt{\Sigma}\setminus\{\wt{z}\})^\fix$ is symplectomorphic to
  $\Sym^g(\Sigma\setminus\{z\})$, and the symplectomorphism takes the
  fixed sets $(\TT_{\wt{\alpha}}^\fix,\TT_{\wt{\beta}}^\fix)$ of the
  Lagrangian tori to the Lagrangian tori $\TT_\alpha$ and $\TT_\beta$.
\end{lemma}
\begin{proof}
  The proof is the same as the analogous result for branched double
  covers of genus $0$ multi-pointed Heegaard diagrams for links in
  $S^3$~\cite[Section 4 and Appendix A]{Hendricks12:dcov-localization}.
\end{proof}

\begin{lemma}\label{lem:sym-is-torus}
  Let $\bigvee_{i=1}^k S^1_i$ be a bouquet of circles. Choose
  coordinates on each $S^1_i$ such that the wedge point is $1\in S^1\subset\CC$. Then $\Sym^r(\bigvee_{i=1}^k S^1_i)$ deformation retracts onto its subspace
  \[
    \Bigl\{(z_1,\dots,z_k) \in \prod_{i=1}^k S^1_i \mid \text{at most } r\text{ coordinates satisfy } z_i \neq 1\Bigr\}.
  \]
  In particular, if $r\geq k$, $\Sym^r(\bigvee_{i=1}^k S^1_i)$ is
  homotopy equivalent to the
  $k$-torus $\prod_{i=1}^k S^1_i$, while if $r < k$, then
  $\Sym^r(\bigvee_{i=1}^k S^1_i)$ is homotopy equivalent to the $r$-skeleton of the $k$-torus
  $\prod_{i=1}^k S^1_i$ with respect to the standard product CW decomposition of the
  torus.
\end{lemma}

\begin{proof}
  The map $\Sym^r(S^1) \rightarrow S^1$ given by multiplication
  $\{z_1,\dots, z_r\} \mapsto z_1\cdots z_r$ is a homotopy equivalence
  (see, e.g., the proof of~\cite[Lemma 5.1]{Hendricks12:dcov-localization} or, for
  the essence of the argument,~\cite[Example 4K.4]{Hatcher02:book}). Work of Ong~\cite{Ong:bouquets} (see also~\cite[Lemma 5.1]{Hendricks12:dcov-localization}) shows that this map can be used to construct the desired deformation retract from $\Sym^r(\bigvee_{i=1}^k S^1_i)$ to the $r$-skeleton of the torus.
\end{proof}

\begin{corollary} \label{corollary:integral-iso}
  Given a complex vector bundle $E\to \Sym^g(\Sigma\setminus\{z\})$,
  the Chern character map
  $\chern\co \Iso(E,E)\to H^{\mathrm{odd}}(\Sym^g(\Sigma\setminus \{z\});\QQ)$
  (Section~\ref{sec:K1}) is injective with image
  $H^{\mathrm{odd}}(\Sym^g(\Sigma\setminus \{z\});\ZZ)$ and hence
  induces an isomorphism $\chern\co \Iso(E,E)\to H^{\mathrm{odd}}(\Sym^g(\Sigma\setminus \{z\});\ZZ)$.
\end{corollary}
\begin{proof}
  If $X$ is a wedge sum of spheres then the
  Chern character map is an isomorphism $K^0(X)\to
  H^{\mathrm{even}}(X)$~\cite[pp. 212]{May:concise}. So, since the Chern character map under
  consideration is induced from the usual Chern character map on the
  suspension of $X$, the result follows from
  Lemma~\ref{lem:sym-is-torus} and the fact that the suspension of a
  skeleton of a torus is a wedge sum of spheres.
\end{proof}

Given a doubly-pointed Heegaard diagram $(\Sigma,\alphas,\betas,z,w)$
for a nullhomologous knot $K$, with branched double cover diagram
$(\wt{\Sigma},\wt{\alphas},\wt{\betas},\wt{z},\wt{w})$,
Large~\cite[Proposition 10.2]{Large} constructed a stable tangent-normal
isomorphism
\begin{equation}
  \label{eq:Large-tan-norm}
  \Phi_1\co (T\Sym^{2g}(\wt{\Sigma}\setminus\{\wt{z},\wt{w}\})^\fix,
  T\TT_{\wt{\alpha}}^\fix, T\TT_{\wt{\beta}}^\fix)\stackrel{\cong}{\lra}
  (N\Sym^{2g}(\wt{\Sigma}\setminus\{\wt{z},\wt{w}\})^\fix,N\TT_{\wt{\alpha}}^\fix, N\TT_{\wt{\beta}}^\fix).
\end{equation}
Eventually, we will modify $\Phi_1$ so that it extends over
$\{w\}\times\Sym^{g-1}(\Sigma)$, without changing $\Phi_1$ on
$\TT_{\wt{\alpha}}^\fix$ and $\TT_{\wt{\beta}}^\fix$ (up to homotopy). As
a first step we have:

\begin{lemma} \label{lemma:bigiso}
  There is a stable isomorphism of complex vector bundles
  \[
    \Phi_2 \co T\Sym^{2g}(\wt{\Sigma}\setminus\{\wt{z}\})^\fix\stackrel{\cong}{\longrightarrow}
    N\Sym^{2g}(\wt{\Sigma}\setminus\{\wt{z}\})^\fix.
  \]
\end{lemma}
\begin{proof}
  Let $E$ be a disk in $\Sigma$ containing $z$ and $w$, so
  that $\Sigma \setminus E$ is a deformation retract of $\Sigma
  \setminus \{z\}$. Let $Y$ be the image of $\Sym^g(\Sigma \setminus
  E)$ in $\Sym^{2g}(\wt{\Sigma}\setminus \{\wt{z},\wt{w}\})^\fix$. Large's
  isomorphism $\Phi_1$
  restricts to an isomorphism $TY \simeq NY$. Since $Y$ is a deformation
  retract of $\wt{\Sigma}\setminus\{\wt{z}\}$, this implies the existence of
  the isomorphism $\Phi_2$.
\end{proof}
Note that, in the proof of Lemma~\ref{lemma:bigiso}, since $E$ may
intersect the $\alpha$- and $\beta$-curves, we have no control over
$\Phi_2$ on $T\TT_{\wt{\alpha}}^\fix$ and $T\TT_{\wt{\beta}}^\fix$.

\begin{remark}
  One can alternately prove Lemma~\ref{lemma:bigiso} by using Macdonald's computation of the Chern classes
  of symmetric products of surfaces~\cite{Macdonald:symmetric}, along
  with the fact that over spaces with torsion-free cohomology the
  Chern classes of a vector bundle determine its stable isomorphism
  class.
\end{remark}

\begin{lemma}\label{lemma:works-for-cohomology}
  Let $V$ be a closed tubular neighborhood of $\{w\}\times
  \Sym^{g-1}(\Sigma\setminus\{z\})\subset \Sym^g(\Sigma\setminus\{z\})$.
  Consider the commutative diagram
  \[
    \xymatrix{
      G\ar[dr] & & \\
      H^*(\Sym^g(\Sigma\setminus\{z\}))\ar[r]\ar[d] &
      H^*(\Sym^g(\Sigma\setminus\{z,w\}))\ar[dr]\ar[d] &
      \\
      H^*(V)\ar[r] & H^*(\bdy V) & H^*(\TT_\alpha)\oplus H^*(\TT_\beta)
    }
  \]
  where $G$ is the kernel of the map
  $H^*(\Sym^g(\Sigma\setminus\{z,w\}))\to H^*(\TT_\alpha)\oplus H^*(\TT_\beta)$ (so the diagonal line is exact).
  Given any class $a\in H^*(\Sym^g(\Sigma\setminus\{z,w\}))$ there is
  a class $b\in G$ so that the image of $a+b$ in $H^*(\bdy V)$ is in
  the image of $H^*(V)$.
\end{lemma}
\begin{proof}
  Let $\gamma\subset \Sigma\setminus\{z,w\}$ be a small circle around
  $w$.  Since $K$ is nullhomologous there is a class
  $c\in H^1(\Sigma\setminus\{z,w\})$ so that
  $c([\alpha_i])=c([\beta_i])=0$ for all $i$ and
  $c([\gamma])=1$. Specifically, since $K$ is nullhomologous, $K$
  bounds a Seifert surface $F$ in $Y$. The Poincar\'e-Lefschetz dual
  $\PD([F])\in H^1(Y\setminus K)$ evaluates to $1$ on a meridian
  of $K$. Since each $\alpha_i$ and $\beta_i$ is nullhomologous in
  $Y\setminus K$ (they bound disks), $\PD([F])$ evaluates to $0$ on
  $[\alpha_i]$ and $[\beta_i]$. Hence, the image of $\PD[F]$ in
  $H^1(\Sigma\setminus\{z,w\})$ is the desired class $c$.

  Projection to $\{w\}\times\Sym^{g-1}(\Sigma\setminus\{z\})$
  gives a homotopy equivalence
  $V\simeq \Sym^{g-1}(\Sigma\setminus\{z\})$. Further, since the restriction of 
  $c_1(T\Sym^g(\Sigma\setminus\{z\}))$ to $\{w\}\times\Sym^{g-1}(\Sigma\setminus\{z\})$ is exactly $c_1(T\Sym^{g-1}(\Sigma\setminus\{z\}))$~\cite[Formula
  14.5]{Macdonald:symmetric}, the normal bundle to
  $\{w\}\times\Sym^{g-1}(\Sigma\setminus\{z\})$ is trivial so
  $\bdy V\cong \Sym^{g-1}(\Sigma\setminus\{z\})\times
  S^1$. (The
  restriction of the cohomology class $\eta\in H^2$ appearing in
  MacDonald's formula to the symmetric product of
  $\Sigma\setminus\{z\}$ vanishes.) From Lemma~\ref{lem:sym-is-torus},
  the cohomology $H^i(\Sym^{g-1}(\Sigma\setminus\{z,w\}))$ vanishes
  for $i>g-1$ and the inclusion map
  $\Sym^{g-1}(\Sigma\setminus\{z,w\}))\into
  \Sym^{g}(\Sigma\setminus\{z,w\}))$ induces an isomorphism on $H^i$
  for $i\leq g-1$. By a small abuse of notation, let $c$ denote the image of the class $c \in H^1(\Sigma\setminus\{z,w\})$ in $H^1(\Sym^{n}(\Sigma\setminus\{z,w\}))$ under this string of isomorphisms for any $n$. We thus have a diagram 
  \[
    \begin{tikzpicture}
      \node at (0,0) (tl)
      {$H^*(\Sym^{g-1}(\Sigma\setminus\{z,w\}))^{\oplus 2}$};
      \node at (7,0) (tr) {$H^*(\Sym^{g}(\Sigma\setminus\{z,w\}))^{\oplus 2}$};
      \node at (0,-1.5) (cl) {$H^*(\Sym^{g-1}(\Sigma\setminus\{z\}))^{\oplus 2}$};
      \node at (7,-1.5) (cr) {$H^*(\Sym^g(\Sigma\setminus\{z,w\}))$};
      \node at (0,-3) (bl) {$H^*(V)^{\oplus 2}$};
      \node at (7,-3) (br) {$H^*(\bdy V)$};
      \draw[->] (cl) to node[left]{\lab{\cong}} (bl);
      \draw[->] (cr) to (br);
      \draw[->] (tr) to node[right]{\lab{(x,y)\mapsto x+c\cup y}} (cr);
      \draw[->] (bl) to node[below]{\lab{(x,y)\mapsto i^*x+c\cup i^*y}} node[above]{\lab{\cong}} (br);
      \draw[->] (cl) to node[left]{\lab{i^*}} (tl);
      \draw[->] (tr) to node[above]{\lab{\approx}} (tl);
    \end{tikzpicture}
  \]
  where the map labeled $\approx$ is an isomorphism for $*<g$ (and the
  target vanishes for $*\geq g$), and the maps $i^*$ are induced by
  the inclusion $\bdy V\into V$ and
  $\Sym^{g-1}(\Sigma\setminus\{z,w\})\into
  \Sym^{g-1}(\Sigma\setminus\{z\})$. 
  

  We claim that if we invert the arrow labeled $\approx$ in the degrees where it is an isomorphism, the diagram commutes. This is a consequence of Lemma~\ref{lem:sym-is-torus}, as follows. Let $S^1_1, \dots, S^1_{2g}$ be a collection of $2g$ circles in $\Sigma \setminus \{z\}$ such that the punctured surface $\Sigma \setminus \{z\}$ deformation retracts onto $\bigvee_{i=1}^{2g} S^1_i$, the surface $\Sigma \setminus \{z,w\}$ deformation retracts onto $(\bigvee_{i=1}^{2g} S^1_i)\vee\gamma$, and the inclusion map $\Sigma \setminus \{z,w\} \hookrightarrow \Sigma \setminus \{z\}$ goes by filling in a disk $D_{\gamma}$ containing $w$ whose boundary is $\gamma$.
  
Lemma~\ref{lem:sym-is-torus} shows that $\Sym^{g}(\Sigma \setminus
\{z\})$ deformation retracts onto the $g$-skeleton
\[
  \Bigl\{(z_1,\dots,z_{2g}) \in  S^1_1\times\cdots\times S^1_{2g} \mid \text{at most } r\text{ coordinates satisfy } z_i \neq 1\Bigr\}
\]
of $\prod_{i=1}^{2g} S^1_i$. However, for this argument we wish to apply a milder deformation retraction, starting with the fact that $\Sigma \setminus \{z\}$ deformation retracts onto $\left(\bigvee_{i=1}^{2g} S^1_i \right)\vee D_{\gamma}$. The same argument as Lemma~\ref{lem:sym-is-torus} shows that
$\Sym^g\left(\left(\bigvee_{i=1}^{2g} S^1_i \right)\vee D_{\gamma}\right)$
deformation retracts onto 
\[
  \Bigl\{(z_1,\dots,z_{2g+1}) \in S^1_1\times\cdots\times S^1_{2g} \times D_{\gamma} \mid \text{at most } r\text{ coordinates satisfy } z_i \neq 1\Bigr\}.
\]
This deformation retraction takes the subspace $\Sym^{g}(\Sigma \setminus \{z,w\})$ onto 
\[
  \Bigl\{(z_1,\dots,z_{2g+1}) \in S^1_1\times\cdots\times S^1_{2g} \times (D_\gamma \setminus \{w\}) \mid \text{at most } r\text{ coordinates satisfy } z_i \neq 1\Bigr\}
\]
which itself deformation retracts onto the $g$-skeleton of $\left(\prod_{i=1}^{2g} S^1_i\right) \times \gamma$. Furthermore, it carries $V$ to the product of the $g-1$ skeleton of $\prod_{i=1}^{2g} S^1_i$ with $D_{\gamma}$, and $\partial V$ onto the product of the $g-1$ skeleton of $\prod_{i=1}^{2g} S^1_i$ with $\gamma$. It is now simple to see from this description of the spaces in terms of tori that the diagram above commutes.

  Write the image of $a$ in $H^*(\bdy V)$ as $i^*a_1+c\cup i^*a_2$ for
  some $a_1,a_2\in H^*(V)$. Let
  $a_i'\in H^*(\Sym^{g-1}(\Sigma\setminus\{z\}))$ be a preimage of
  $a_i$ under the isomorphism. Since the top horizontal map is an
  isomorphism whenever $H^*(\Sym^{g-1}(\Sigma\setminus\{z,w\}))$ (or equivalently $H^*(V)$) is
  non-zero, there are elements
  $\wt{a}_i\in H^*(\Sym^g(\Sigma\setminus\{z,w\}))$ mapping to $i^*a'_i$.
  Take $b=-c\cup \wt{a}_2\in H^*(\Sym^g(\Sigma\setminus\{z,w\}))$.
  Since $c|_{\alpha_i}$ and $c|_{\beta_i}$ vanish, $b$ lies in the
  kernel $G$. The image of
  $a+b\in H^*(\Sym^g(\Sigma\setminus\{z,w\}))$ in $H^*(\bdy V)$ is the
  same as the image of $a_1\in H^*(V)$. This proves the result.
\end{proof}

\begin{proof}[Proof of Proposition~\ref{prop:trivialization}]
  The composition
  \[
    \Phi_2^{-1} \circ \Phi_1\co
    T\Sym^{2g}(\wt{\Sigma}\setminus\{\wt{z},\wt{w}\})^\fix \rightarrow
    T\Sym^{2g}(\wt{\Sigma}\setminus\{\wt{z},\wt{w}\})^\fix
  \]
  is an element of
  \[
    \Iso(T\Sym^{2g}\bigl(\wt{\Sigma}\setminus\{\wt{z},\wt{w}\})^\fix,
    T\Sym^{2g}(\wt{\Sigma}\setminus\{\wt{z},\wt{w}\})^\fix\bigr).
  \]
  Identify
  $\Sym^g(\Sigma \setminus \{z,w\})$ with
  $\Sym^{2g}(\wt{\Sigma}\setminus\{\wt{z},\wt{w}\})^\fix$ as in
  Lemma~\ref{lem:fixed-set} and let
  $a = \chern[\Phi_2^{-1} \circ \Phi_1] \in
  H^{\mathrm{odd}}(\Sym^g(\Sigma \setminus \{z,w\}))$ (see
  Section~\ref{sec:K1} and
  Remark~\ref{rem:cx-is-symp}). By Lemma
  \ref{lemma:works-for-cohomology}, there exists
  $b \in H^{\mathrm{odd}}(\Sym^g(\Sigma \setminus \{z,w\}))$ such that
  $b$ is in the kernel of the map
  $H^*(\Sym^g(\Sigma \setminus \{z,w\})) \rightarrow
  H^*(\TT_\alpha)\oplus H^*(\TT_\beta)$ and the image of $a+b$
  in $H^*(\partial V)$ is in the image of the map
  $H^*(V) \rightarrow H^*(\partial V)$. By Corollary
  \ref{corollary:integral-iso}, $b = \chern[\Phi_3]$ for some
  \[
    \Phi_3 \in \Iso(T\Sym^{2g}(\wt{\Sigma}\setminus\{\wt{z},\wt{w}\})^\fix,
    T\Sym^{2g}(\wt{\Sigma}\setminus\{\wt{z},\wt{w}\})^\fix).
  \]
  Functoriality of the Chern character implies that
  $\chern[\Phi_3|_{T\TT_{\wt{\alpha}}^\fix\otimes \mathbb C}]=0$. Hence, the restriction
  $\Phi_3|_{T\TT_{\wt{\alpha}}^\fix\otimes \mathbb C}$ is stably homotopic to
  the identity isomorphism. Likewise,
  $\Phi_3|_{T\TT_{\wt{\beta}}^\fix\otimes \mathbb C}$ is stably homotopic to
  the identity isomorphism.

  Consider
  \[
    \Phi_2^{-1} \circ \Phi_1 \circ \Phi_3:
    T\Sym^{2g}(\wt{\Sigma}\setminus\{\wt{z},\wt{w}\})^\fix \rightarrow
    T\Sym^{2g}(\wt{\Sigma}\setminus\{\wt{z},\wt{w}\})^\fix.
  \]
  Since $\chern[\Phi_2^{-1} \circ \Phi_1 \circ \Phi_3] =
  \chern[\Phi_2^{-1} \circ \Phi_1] + \chern [\Phi_3] = a+b$ and the
  Chern character is functorial, we see that $\chern[(\Phi_2^{-1} \circ
  \Phi_1 \circ \Phi_3)|_{\partial V}]$ is the image of $a+b$ in
  $H^*(\partial V)$ and therefore lies in the image of the bottom
  horizontal map in the following commutative diagram:
  \[
  \begin{tikzpicture}
    \node at (0,0) (TV) {$\Iso(TV,TV)$};
    \node at (4,0) (bdyV) {$\Iso(TV|_{\bdy V},TV|_{\bdy V})$};
    \node at (0,-2) (HV) {$H^{\mathrm{odd}}(V)$};
    \node at (4,-2) (HbdyV) {$H^{\mathrm{odd}}(\bdy V)$};
    \node at (6,1) (tr) {$(\Phi_2^{-1}\circ\Phi_1\circ\Phi_3)|_{\bdy V}$};
    \node[rotate=-120] at (5.25,.5) (in2) {$\in$};
    \node at (-1.5,-3) (bl) {$(a+b)|_V$};
    \node[rotate=60] at (-1,-2.5) (in3) {$\in$};
    \node at (6,-3) (br) {$\chern\bigl((\Phi_2^{-1}\circ\Phi_1\circ\Phi_3)|_{\bdy V}\bigr).$};
    \node[rotate=120] at (5.25,-2.5) (in4) {$\in$};
    \draw[->] (TV) to (bdyV);
    \draw[->] (TV) to node[left]{\lab{\chern}} node[right]{\lab{\cong}} (HV);
    \draw[->] (bdyV) to node[left]{\lab{\chern}} node[right]{\lab{\cong}} (HbdyV);
    \draw[->] (HV) to (HbdyV);
    \draw[|->] (bl) to (br);
    \draw[|->] (tr) to (br);
  \end{tikzpicture}
\]
Corollary~\ref{corollary:integral-iso} implies that the vertical maps
in this diagram are isomorphisms, so the isomorphism $(\Phi_2^{-1} \circ \Phi_1 \circ \Phi_3)|_{\partial V}$ extends over $V$. There is therefore an extension 
\[
\Phi_4 \co T\Sym^{2g}(\wt{\Sigma}\setminus\{\wt{z}\})^\fix \rightarrow T\Sym^{2g}(\wt{\Sigma}\setminus\{\wt{z}\})^\fix
\]
of $\Phi_2^{-1}\circ \Phi_1 \circ \Phi_3$. Our final isomorphism $\Phi_5$ is the composition
\[
  \Phi_5\coloneqq \Phi_2 \circ \Phi_4:
  T\Sym^{2g}(\wt{\Sigma}\setminus\{\wt{z}\})^\fix \rightarrow N\Sym^{2g}(\wt{\Sigma}\setminus\{\wt{z}\})^\fix.
\]
This map $\Phi_5$ agrees with $\Phi_1 \circ \Phi_3$ away from the divisor
$\{w\}\times\Sym^{g-1}(\Sigma)$. Since the restriction of $\Phi_3$ to
$T\TT_{\wt{\alpha}}^\fix\otimes \mathbb C$ is homotopic to the identity
and there is a homotopy of Lagrangian subbundles from
$\Phi_1(T\TT_{\wt{\alpha}}^\fix)$ to $N\TT_{\wt{\alpha}}^\fix$, there is a
homotopy of Lagrangian subbundles from $\Phi_5(T\TT_{\wt{\alpha}}^\fix)$
to $N\TT_{\wt{\alpha}}^\fix$, and similarly for
$T\TT_{\wt{\beta}}^\fix$. Therefore the map $\Phi_5$ is the desired stable tangent-normal isomorphism
  \[
    \bigl(T\Sym^{2g}(\wt{\Sigma}\setminus\{\wt{z}\})^\fix,T\TT_{\wt{\alpha}}^\fix,T\TT_{\wt{\beta}}^\fix\bigr)\cong
    \bigl(N\Sym^{2g}(\wt{\Sigma}\setminus\{\wt{z}\})^\fix,N\TT_{\wt{\alpha}}^\fix,N\TT_{\wt{\beta}}^\fix\bigr).\qedhere
  \]
\end{proof}

\section{Proof of the main theorem}\label{sec:proof}
In this section, we prove Theorem~\ref{thm:main}.  We begin with the
simplest version of the spectral sequence, and then prove a 
$\SpinC$-refined statement in Section~\ref{sec:spinc} and the
generalization from knots to links in Section~\ref{sec:links}.

\begin{theorem}\label{thm:main-knots}
  Let $Y$ be a closed $3$-manifold and $K\subset Y$ an oriented nullhomologous knot 
  with Seifert surface $F$. Let $\pi\co \Sigma(Y,K)\to Y$ be the double cover branched
  along $K$ induced by the Seifert surface $F$.
  Then, there is a spectral sequence with $E^1$-page given by
  \[
    \HFa(\Sigma(Y,K)) \otimes \FF_2[[\theta, \theta^{-1}]
  \]
  converging to
  \[
    \HFa(Y) \otimes \FF_2[[\theta, \theta^{-1}].
  \]
  In particular, 
  \[
    \dim\HFa(\Sigma(Y,K)) \geq \dim\HFa(Y).
  \]
\end{theorem}
\begin{proof}
  Fix $\HD=(\Sigma_g,\alphas,\betas,z,w)$ a weakly admissible doubly-pointed Heegaard
  diagram for a nullhomologous knot $K$ in $Y$ and let $\wt{\HD}=(\wt{\Sigma}_{2g},\wt{\alphas},\wt{\betas},\wt{z},\wt{w})$
  denote a doubly-pointed Heegaard diagram for $(\Sigma(Y,K),\wt{K})$
  obtained by taking the branched double cover of $\HD$. By
  Proposition~\ref{prop:symplectic-hypotheses} below, $\wt{\HD}$ is also
  weakly admissible. By Proposition~\ref{prop:trivialization}, there is a stable tangent-normal isomorphism 
  \[
    \bigl(T\Sym^{2g}(\wt{\Sigma}\setminus\{\wt{z}\})^\fix,T\TT_{\wt{\alpha}}^\fix,T\TT_{\wt{\beta}}^\fix\bigr)\cong
    \bigl(N\Sym^{2g}(\wt{\Sigma}\setminus\{\wt{z}\})^\fix,N\TT_{\wt{\alpha}}^\fix,N\TT_{\wt{\beta}}^\fix\bigr).
  \]
  By Proposition~\ref{prop:symplectic-hypotheses} again, the remaining
  hypotheses of Theorem~\ref{thm:Large} are satisfied.  So,
  Theorem~\ref{thm:Large} implies the result.
\end{proof}

\begin{proposition}\label{prop:symplectic-hypotheses}
  Let $\HD=(\Sigma,\alphas,\betas,z,w)$ be a Heegaard diagram for
  a nullhomologous knot $K$ in a closed $3$-manifold $Y$ and let
  $\wt{\HD}=(\wt{\Sigma},\wt{\alphas},\wt{\betas},\wt{z},\wt{w})$ be a branched
  double cover of $\HD$.
  Assume that $\HD$ is weakly admissible for all $\SpinC$-structures.
  Then $(\wt{\Sigma},\wt{\alphas},\wt{\betas},\wt{z},\wt{w})$ is weakly admissible for all $\SpinC$-structures. Further, there is a choice of
  symplectic form on $\Sym^g(\wt{\Sigma}\setminus\{\wt{z}\})$
  satisfying hypotheses~\ref{item:symp-hyp} and~\ref{item:tau-hyp}
  from Theorem~\ref{thm:Large} (and inducing the polarization data
  studied in Section~\ref{sec:normalIso}).
\end{proposition}
\begin{proof}
  Weak admissibility is equivalent to the existence of
  an area form $\omega$ on $\Sigma$ so that the signed area of every
  periodic domain with multiplicity $0$ at $z$ is
  zero~\cite[Lemma 4.12]{OS04:HolomorphicDisks}. Since $K$ is
  nullhomologous, every periodic domain for
  $(\wt{\Sigma},\wt{\alphas},\wt{\betas})$ with multiplicity $0$ at
  $\wt{z}$ also has multiplicity $0$ at $\wt{w}$, and hence projects
  to a periodic domain in $\Sigma$ with multiplicity $0$ at $z$ (and
  $w$). Hence, the pullback $\wt{\omega}$ of $\omega$ (smoothed out at $\wt{z}$ and
  $\wt{w}$) has the property that every periodic domain with
  multiplicity $0$ at $\wt{z}$ has signed area $0$. In particular,
  $(\wt{\Sigma},\wt{\alphas},\wt{\betas},\wt{z})$ is also weakly
  admissible for all $\SpinC$-structures.

  Perutz's techniques~\cite[Section 7]{Perutz07:HamHand}, as
  applied by Hendricks to the case of punctured Heegaard
  surfaces~\cite[Section 4]{Hendricks12:dcov-localization}, show that
  if $\phi$ is an exhausting function on $\Sigma \setminus \{z\}$ such
  that $\omega = -dd^{\mathbb C}\phi$ and $\wt{\phi}$ is the lift of
  $\phi$ to $\wt{\Sigma} \setminus \{\wt{z}\}$, then there is an
  equivariant smooth exhausting function $\psi$ on
  $\Sym^{2g}(\wt{\Sigma} \setminus \{\wt{z}\})$ which agrees with
  $\wt{\phi}^{\times 2g}$ away from a neighborhood of the diagonal. In
  particular, if $\wt{\omega}=-dd^{\mathbb C}\wt{\phi}$ is the
  symplectic form on $\wt{\Sigma} \setminus \{\wt{z}\}$, then
  $-dd^{\mathbb C}\psi$ is an exact equivariant symplectic form on
  $M=\Sym^{2g}(\wt{\Sigma} \setminus \{\wt{z}\})$ which agrees with
  $\wt{\omega}^{\times 2g}$ away from a neighborhood of the
  diagonal. This shows that $M$ is an exact symplectic manifold and
  convex at infinity. Further, if $\lambda = -d^{\mathbb C}\wt{\phi}$
  then $-dd^{\mathbb C}\psi$ has a primitive $-d^{\mathbb C}\psi$ that agrees with
  $\lambda^{\times 2g}$ away from the diagonal. 

  To establish that $L_0 = \TT_{\wt{\alpha}}$ and
  $L_1 = \TT_{\wt{\beta}}$ are exact Lagrangians in $M$, we first
  check that the curves $\wt{\alpha}_i$ and $\wt{\beta}_j$ are exact with
  respect to a suitable primitive of $\wt{\omega}$ in $\wt{\Sigma} \setminus \{\wt{z}\}$.
  Consider the primitive $\lambda=-d^{\CC}\wt{\phi}$ of
  $\wt{\omega}$. We will adjust $\lambda$ on $\wt{\Sigma}\setminus\{\wt{z}\}$ so that for all $i$,
  $\int_{\wt{\alpha}_i}\lambda=\int_{\wt{\beta}_i}\lambda=0$, and then
  adjust $-d^\CC\psi$ correspondingly on $\Sym^{2g}(\wt{\Sigma}\setminus\{\wt{z}\})$. Reordering
  the $\wt{\beta}_i$, arrange that
  $[\wt{\alpha}_1],\dots,[\wt{\alpha}_{2g}],[\wt{\beta}_1],\dots,[\wt{\beta}_k]\in
  H_1(\wt{\Sigma};\QQ)$ are linearly independent and
  \begin{equation}\label{eq:span}
    [\wt{\beta}_{k+1}],\dots,[\wt{\beta}_{2g}]\in\Span([\wt{\alpha}_1],\dots,[\wt{\alpha}_{2g}],[\wt{\beta}_1],\dots,[\wt{\beta}_k])\subset H_1(\wt{\Sigma};\QQ).
  \end{equation}
  There is a cohomology class
  $[a]\in H^1(\wt{\Sigma};\RR)$ so that for all $i=1,\dots,2g$,
  $\langle [a], [\wt{\alpha}_i]\rangle=\int_{\wt{\alpha}_i}\lambda$, and for
  $i=1,\dots,k$, $\langle [a],[\wt{\beta}_i]\rangle =
  \int_{\wt{\beta}_i}\lambda$. Choose a closed $1$-form $a$ representing
  $[a]$ and let $\lambda'=\lambda-a$. Then $\lambda'$ is still a
  primitive of $\wt{\omega}$ and
  $\int_{\wt{\alpha}_i}\lambda'=\int_{\wt{\beta}_j}\lambda'=0$ for $1\leq i\leq 2g$
  and $1\leq j\leq k$. We claim that in fact $\int_{\wt{\beta}_j}\lambda'=0$
  for $j=k+1,\dots,2g$ as well. By Equation~\eqref{eq:span} there is a
  periodic domain $P$ with boundary
  \[
    \bdy P = m_1[\wt{\alpha}_1]+\cdots+m_{2g}[\wt{\alpha}_{2g}]+n_1[\wt{\beta}_1]+\cdots+n_k[\wt{\beta}_k]+p[\wt{\beta}_j]
  \]
  for some $m_1,\dots,m_{2g},n_1,\dots,n_k,p\in\ZZ$, $p\neq 0$. By Stokes'
  theorem,
  \[
    p\int_{\wt{\beta}_j}\lambda' = \int_P\wt{\omega}-m_1\int_{\wt{\alpha}_1}\lambda'-\cdots-n_k\int_{\wt{\beta}_k}\lambda',
  \]
  but by construction every term on the right-hand side vanishes.

  Now, let $[b]\in H^1(\Sym^g(\wt{\Sigma}\setminus\{\wt{z}\});\RR)$ be the image
  of the class $[a]$ under the isomorphism
  $H^1(\Sym^{2g}(\wt{\Sigma}\setminus\{\wt{z}\});\RR)\cong
  H^1(\wt{\Sigma}\setminus\{\wt{z}\};\RR)$ induced by the inclusion
  $\wt{\Sigma}\into\Sym^{2g}(\wt{\Sigma})$, and let $b$ be a closed
  $1$-form representing $[b]$. Then $-d^{\CC}\psi-b$ is a primitive
  for the symplectic form on
  $\Sym^{2g}(\wt{\Sigma}\setminus\{\wt{z}\})$ and, from the
  computation in the previous paragraph, the restriction of
  $-d^{\CC}\psi-b$ to $\TT_{\wt{\alpha}}$ and $\TT_{\wt{\beta}}$
  is exact.  This
  concludes the proof.
\end{proof}

\subsection{The \texorpdfstring{$\SpinC$}{spin-c} refinement}\label{sec:spinc}
In this section, we will refine Theorem~\ref{thm:main-knots} to
respect $\SpinC$ structures.
First, we must discuss $\SpinC$ structures on branched covers.  

\begin{definition}\label{def:pullback-spinc}
  Let $\spinc$ be a $\SpinC$-structure on $Y$ and
  $\pi\co(\Sigma(Y,K),\wt{K})\to (Y,K)$ be a double cover branched
  along a nullhomologous knot $K$.
  The {\em pullback} $\SpinC$-structure $\pi^*\spinc$ is characterized
  as follows. If $\wt{K}=\pi^{-1}(K)$ denotes the double point set then on
  $\Sigma(Y,K)\setminus\nbd(\wt{K})$, the map $\pi$ is a local
  diffeomorphism, so
  $T\bigl(\Sigma(Y,K)\setminus\nbd(\wt{K})\bigr)\cong \pi^*T\bigl( Y \setminus \nbd(K) \bigr)$. Thus,
  $\spinc \in \SpinC(Y)$ induces a $\SpinC$-structure $\pi^*\spinc$ on
  $\Sigma(Y,K)\setminus\nbd(\wt{K})$. The obstruction to extending
  $\pi^*\spinc|_{\bdy\nbd(\wt{K})}$ over $\nbd(\wt{K})$ is
  $c_1(\pi^*\spinc|_{\bdy\nbd(\wt{K})})=\pi^*c_1(\spinc|_{\bdy\nbd(\wt{K})})$,
  which is the pullback of the obstruction
  to extending $\spinc$ over $\nbd(K)$
  and hence vanishes. Any two extensions differ by a multiple of
  $\PD[\wt{K}]=0$, so the extension of $\pi^*\spinc$ to all of
  $\Sigma(Y,K)$ is unique.

  For the branched double cover of an (oriented) nullhomologous link $L$ where some components are
  homologically essential, the uniqueness step above fails. For links,
  define $\pi^*\spinc$ as follows. Identify a neighborhood of $L$ with
  $D^2\times L$ so that the Seifert surface is given by $[0,1)\times
  \{0\}\times L$. Choose a vector field $v$ on $Y$
  representing $\spinc$, and so that in this neighborhood $v$ is
  given by $\partial/\partial\theta$, where $\theta$ is a coordinate
  on $L$. In particular, $v$ is positively tangent to
  $L$. From the construction of the branched double cover, there is an
  induced vector field $\wt{v}$ on $\Sigma(Y,L)$ so that on
  $\Sigma(Y,L)\setminus\wt{L}$, $d\pi(\wt{v})=v$, and $\wt{v}$ is
  positively tangent to $\wt{L}$. Then $\pi^*\spinc$ is the
  $\SpinC$-structure represented by $\wt{v}$.
\end{definition}
It is immediate from the construction that, for knots, these two
definitions of $\pi^*\spinc$ agree. It follows from
Proposition~\ref{prop:knotification} below that for links the second
construction is independent of the choice of $v$ representing
$\spinc$. It also follows that reversing the orientation of all
components of $L$ gives the same map $\pi^*$ on $\SpinC$-structures.

We note next that the definition of pullback $\SpinC$ structures
behaves well with respect to the association of $\SpinC$ structures to
intersection points in Heegaard diagrams.  Fix
$\HD=(\Sigma_g,\alphas,\betas,z,w)$ a doubly-pointed Heegaard diagram
for a nullhomologous knot $K$ in $Y$ and let
$\wt{\HD}=(\wt{\Sigma},\wt{\alphas},\wt{\betas},\wt{z},\wt{w})$ be a
branched double cover of $\HD$, which is a doubly-pointed Heegaard
diagram for $(\Sigma(Y,K),\wt{K})$.  Recall that Ozsv\'ath-Szab\'o~\cite{OS04:HolomorphicDisks}
gave an association
$\spinc_z\co \TT_\alpha \cap \TT_\beta \to \SpinC(Y)$.  For
$x \in \TT_\alpha \cap \TT_\beta$, we will sometimes write $\wt{x}$ for the
intersection point $\pi^{-1}(x)$ in
$\TT_{\wt{\alpha}} \cap \TT_{\wt{\beta}}$.
\begin{lemma}\label{lem:pullback-spinc-intersections}
  Let $K$ be a nullhomologous knot in $Y$. Then for
  $x \in \TT_\alpha \cap \TT_\beta$, we have
  $\pi^*(\spinc_z(x)) = \spinc_{\wt{z}}(\pi^{-1}(x))$.
\end{lemma}
\begin{proof}
  Choose a Morse function $f$ on $(Y,K)$ compatible with the
  doubly-pointed Heegaard diagram $(\Sigma, \alphas, \betas, z, w)$.
  Represent $\spinc_z(x)|_{Y\setminus \nbd(K)}$ by a non-vanishing
  vector field by modifying $\nabla f$ on $Y \setminus\nbd(K)$ in a
  neighborhood of the trajectories of $\nabla f$ through $x$.  Consider
  $\pi^* f=f\circ\pi$ and the induced homology class of vector field
  on $\Sigma(Y,K) \setminus \nbd(\widetilde{K})$.  This class is precisely the
  $\SpinC$ structure on $\Sigma(Y,K) \setminus \nbd(\widetilde{K})$
  corresponding to $\tilde{x}$.  Now, define a $\SpinC$ structure on
  $\Sigma(Y,K)$ by extending over $\Sigma(Y,K) \setminus \nbd(\widetilde{K})$.
  As discussed in Definition~\ref{def:pullback-spinc}, the extension
  is unique because $K$ is nullhomologous.
  Hence, this $\SpinC$ structure is exactly
  $\spinc_{\wt{z}}(\pi^{-1}(x))$.  However, this $\SpinC$-structure is
  also $\pi^*(\spinc_z(x))$ as constructed in
  Definition~\ref{def:pullback-spinc}.
\end{proof} 

\begin{remark}\label{rmk:pullback-spinc}  By Lemma~\ref{lem:pullback-spinc-intersections}, if we change the intersection point $x$ for $Y$ without changing the corresponding $\SpinC$ structure on $Y$, then the lifted elements represent the same $\SpinC$ structure on $\Sigma(Y,K)$.  Another way to see this is as follows.  Given a Whitney disk $u \in \pi_2(x,y)$ in $Sym^g(\Sigma \setminus \{z\})$, this naturally induces a Whitney disk $\wt{u} \in \pi_2(\wt{x}, \wt{y})$ in $Sym^{\wt{g}}(\wt{\Sigma} \setminus \{\wt{z}\})$ by $\widetilde{u}(q) = \pi^{-1}(u(q))$.  
\end{remark}

The alternative description of pullback $\SpinC$ structures described in the proof of Lemma~\ref{lem:pullback-spinc-intersections} is also a useful viewpoint for studying the connection between $\SpinC$ structures and cohomology classes.    
\begin{lemma}\label{lem:pullback-spinc-behaves}
  For $K\subset Y$ a nullhomologous knot, the pullback $\SpinC$ structure satisfies
  \begin{align}
	\pi^*\overline{\spinc} &= \overline{\pi^*\spinc} \label{pullback:conjugate} \\
    \pi^*(\spinc+a)&=\pi^*(\spinc)+\pi^*(a) \label{pullback:affine} \\
    c_1(\pi^*\spinc)&=\pi^*c_1(\spinc) \label{pullback:chern}.
  \end{align}
  for any $\spinc\in\SpinC(Y)$ and $a\in H^2(Y)$.
\end{lemma}
\begin{proof}
  \eqref{pullback:conjugate} Recall that if $v$ is a non-vanishing
  vector field corresponding to a $\SpinC$ structure $\spinc$, then
  $-v$ corresponds to $\overline{\spinc}$.  So, the claim follows
  easily from Definition~\ref{def:pullback-spinc}, since if $v$
  corresponds to $\spinc$ on $Y$, then $v |_{Y \setminus \nbd(K)}$
  corresponds to $\spinc |_{Y \setminus \nbd(K)}$ on $Y \setminus \nbd(K)$, and
  $\pi^*v |_{\Sigma(Y,K) \setminus \nbd(\wt{K})}$ corresponds to
  $\pi^* \spinc |_{\Sigma(Y,K) \setminus \nbd(\wt{K})}$.

  \eqref{pullback:affine} This is equivalent to showing that $\pi^*(\spinc' - \spinc) = \pi^*\spinc' - \pi^*\spinc$.  Let $\spinc$ and $\spinc'$ be represented by $x,x' \in \TT_\alpha \cap \TT_\beta$ respectively, so 
\[
\spinc_z(x') - \spinc_z(x) = PD[\epsilon(x,x')],
\]
and 
\[
\spinc_{\wt{z}}(\wt{x}') - \spinc_{\wt{z}}(\wt{x}) = PD[\epsilon(\wt{x},\wt{x}')].
\]
The transfer map $\pi^!$ sends $\epsilon(x,x')$ to
$\epsilon(\wt{x},\wt{x}')$, i.e., $\pi^!\epsilon(x,x') =
\epsilon(\wt{x},\wt{x}')$. (If we represent $\epsilon(x,x')$ by a $1$-manifold in $\Sigma\setminus\{z,w\}$ then $\pi^!\epsilon(x,x')$ is the total preimage of that $1$-manifold.) It follows that 
\begin{align*}
\pi^*(\spinc_z(x') - \spinc_z(x)) &= \pi^*PD[\epsilon(x,x')] \\
&= PD[\pi^!\epsilon(x,x')] \\
&= PD[\epsilon(\wt{x},\wt{x}')] \\
&= \spinc_{\wt{z}}(\wt{x}') - \spinc_{\wt{z}}(\wt{x}) \\
&= \pi^*(\spinc_z(x')) - \pi^*(\spinc_z(x)),
\end{align*}
by Lemma~\ref{lem:pullback-spinc-intersections}. 

\eqref{pullback:chern} Recall that the first Chern class of a $\SpinC$
structure $\mathfrak{t}$ on a closed $3$-manifold can be computed by
$\mathfrak{t} - \overline{\mathfrak{t}}$.  So, the claim follows from Equations~\eqref{pullback:conjugate} and~\eqref{pullback:affine}.  
\end{proof}

We are now ready to state the $\SpinC$-refinement of Theorem~\ref{thm:main}.
\begin{proposition}\label{prop:spinc}
  Let $Y$ be a closed, connected, oriented $3$-manifold, $K \subset Y$ a nullhomologous
  knot, and $\spinc$ a $\SpinC$-structure on $Y$.
  Then, the spectral sequence from Theorem~\ref{thm:main-knots} splits
  along $\tau$-invariant $\SpinC$-structures on $\Sigma(Y,K)$. In particular, there is an inequality
  \[
    \dim \HFa(\Sigma(Y,K), \pi^*\spinc) \geq \sum_{\pi^*\spinc' = \pi^*\spinc} \dim \HFa(Y,\spinc').  
  \]
\end{proposition}
\begin{proof}
Choose a doubly-pointed Heegaard diagram $(\Sigma, \alphas, \betas, z,
w)$ for $K \subset Y$ which is
weakly admissible for all $\SpinC$-structures.  As before, the fixed point sets of the
$\ZZ/2\ZZ$-action on $(\Sym^{2g}(\wt{\Sigma} \setminus \{\wt{z}\}),
\TT_{\wt{\alpha}}, \TT_{\wt{\beta}})$ are identified with $(\Sym^g(\Sigma \setminus \{\wt{z}\}), \TT_{\alpha}, \TT_{\beta})$.  Under this identification, the map 
\[
\iota_* \co \pi_0 P(\TT_{\alpha}, \TT_{\beta}) \to \pi_0 P(\TT_{\wt{\alpha}}, \TT_{\wt{\beta}})
\]
from Section~\ref{sec:Large} sends the constant path $[x]$ associated
to a point $x \in \TT_{\alpha} \cap \TT_{\beta}$ to the constant path
$[\pi^{-1}(x)]$ associated to the point $\pi^{-1}(x) \in \TT_{\wt{\alpha}} \cap \TT_{\wt{\beta}}$.   

Recall that two elements $x,y \in \TT_\alpha \cap \TT_\beta$
have $[x],[y]$ in the same path component in $P(\TT_\alpha, \TT_\beta)$
(inside $\Sym^g(\Sigma \setminus \{z\})$) if
and only if $\spinc_z(x) = \spinc_z(y)$~\cite[Section
2]{OS04:HolomorphicDisks}.   Similarly, two elements
$\wt{x},\wt{y}\in\TT_{\wt{\alpha}} \cap \TT_{\wt{\beta}}$ have
$[\wt{x}],[\wt{y}]$ in the same path
component of $P(\TT_{\wt{\alpha}}, \TT_{\wt{\beta}})$ if and only if
$\spinc_{\wt{z}}(\wt{x})=\spinc_{\wt{z}}(\wt{y})$. Finally, by
Lemma~\ref{lem:pullback-spinc-intersections},
$\pi^*\spinc_z(x) = \spinc_{\wt{z}}(\pi^{-1}(x))$.  Putting this all
together, if an element of $\pi_0 P(\TT_{\alpha}, \TT_{\beta})$
corresponds to $\spinc$, then the image under $\iota_*$ corresponds to
$\pi^*\spinc$.  (See Remark~\ref{rmk:pullback-spinc} for an alternate
viewpoint.)

Thus,
Proposition~\ref{prop:Large-decomp} (together with
Propositions~\ref{prop:trivialization} and~\ref{prop:symplectic-hypotheses}) implies the desired splitting of spectral sequences and inequality 
\[
  \dim \HFa(\Sigma(Y,K), \pi^*\spinc) \geq \sum_{\pi^*\spinc' = \pi^*\spinc}  \dim \HFa(Y,\spinc').\qedhere
\]
\end{proof}

\subsection{From knots to links}\label{sec:links}
In this section, we use Ozsv\'ath-Szab\'o's knotification procedure to
deduce Theorem~\ref{thm:main} for links with an arbitrary number of components
from Proposition~\ref{prop:spinc}.

Suppose $L \subset Y$ has two
components $L_1, L_2$. Let $B_i$ be a ball intersecting $L_i$ in a trivial
arc $A_i$. Note that $Y \# S^2 \times S^1$ can be produced by identifying the
boundary components of $Y \setminus (B_1 \cup B_2)$ so that the endpoints of
$A_1$ and $A_2$ are identified. 
The link
$(L_1\setminus A_1)\cup(L_2\setminus A_2)\subset Y\# S^2\times S^1$ is the
\emph{knotification} of $L$.  More generally, the knotification of an
$\ell$-component link is obtained by doing this process $\ell - 1$ times until a
single component remains in $Y \#_{\ell - 1} S^2 \times S^1$.  We denote the
knotification of $L$ by $\kappa_L$.  It turns out that the knotification
operation behaves well with respect to branched double covers.  Letting
$\spinct$ denote the unique torsion $\SpinC$ structure on
$\#_{\ell - 1} S^2 \times S^1$, we have:

\begin{proposition}\label{prop:knotification}
  Let $L$ be a nullhomologous link in $Y$ with $\ell$ components and let
  $\kappa_L$ be its knotification. Fix a Seifert surface $F$ for $L$,
  let $\pi\co \Sigma(Y,L) \to Y$ denote the corresponding
  double cover of $Y$ branched along $L$, and let $\wt{L}=\pi^{-1}(L)$ be the
  double point set. Then, $\Sigma(Y,L) \#_{\ell - 1} S^2 \times S^1$ is
  homeomorphic to $\Sigma(Y \#_{\ell - 1} S^2 \times S^1, \kappa_L)$ and the
  knotification of $\wt L$ is the preimage of $\kappa_L$.  Furthermore,
  given a $\SpinC$ structure $\spinc$ on $Y$, the pullback of
  $\spinc \# \spinct$ under
  $\pi'\co \Sigma(Y \#_{\ell - 1} S^2 \times S^1 ,\kappa_L) \to Y \#_{\ell - 1}
  S^2 \times S^1$ is $(\pi^*\spinc) \# \spinct$.
\end{proposition}
\begin{proof}
  Recall that the branched double cover of a 3-ball over a trivial arc is
  again a 3-ball and the double point set is a trivial arc.  So, if
  $B_1$ and $B_2$ are small balls around points on two components of
  $L$ then $\pi^{-1}(B_1)$ and $\pi^{-1}(B_2)$ are small balls around
  points on two components of $\wt{L}$, and knotifying $L$ using $B_1$
  and $B_2$ corresponds to knotifying $\wt{L}$ using $\pi^{-1}(B_1)$
  and $\pi^{-1}(B_2)$.

  It remains to identify the $\SpinC$ structures.  For notational
  simplicity, we consider the case of a 2-component link.
  Let $B_3\subset Y$ be the union of $B_1$, $B_2$, and an
  arc connecting them. A
  $\SpinC$-structure $\spinc'$ on $Y\#(S^2\times S^1)$ is determined
  by its restriction to $Y\setminus B_3$ 
  and the evaluation of $c_1(\spinc')$ on $S^2\times\{\pt\}=\bdy B_1$.
  The same remarks hold for $\Sigma(Y,L)\#(S^2\times
  S^1)$. Now, $(\pi^*\spinc)\#\spinct$ and
  $(\pi')^*(\spinc\#\spinct)$ agree on $Y\setminus B_3$ and $\bigl\langle
  c_1\bigl((\pi^*\spinc)\#\spinct\bigr),[S^2]\bigr\rangle=0$. Since
  $(\pi')^*c_1(\spinc\#\spinct)=c_1((\pi')^*(\spinc\#\spinct))$, we have
  $\bigl\langle
  c_1\bigl((\pi')^*(\spinc\#\spinct)\bigr),[S^2]\bigr\rangle=0$ also.  It follows that the
  $\SpinC$-structures $(\pi^*\spinc)\#\spinct$ and $(\pi')^*(\spinc\#\spinct)$ agree.
\end{proof}

\begin{proof}[Proof of Theorem~\ref{thm:main}]
  This is immediate from Propositions~\ref{prop:spinc}
  and~\ref{prop:knotification} and the K\"unneth theorem for $\HFa$ of
  connected sums.
\end{proof}

\begin{remark}   
  The spectral sequence from Theorem~\ref{thm:main} is an invariant of
  $(Y,K)$ in the following sense. Given other choices in its
  construction (Heegaard diagrams, almost complex structures, and so
  on) there is an isomorphism between each page of the resulting
  spectral sequence. This follows from the fact that the spectral
  sequence is isomorphic to Seidel-Smith's spectral
  sequence for equivariant Floer cohomology~\cite[Section
  3.2]{SeidelSmith10:localization} (and hence to the spectral sequence
  one obtains by applying the techniques in~\cite{HLS:HEquivariant} to
  an equivariant Heegaard diagram for the branched double cover) and
  the proof of the analogous result for $\HFKa$~\cite[Corollary 1.10]{HLS:HEquivariant}. On the other
  hand, it is not clear that the isomorphism between the
  $E^\infty$-page of the spectral sequence and
  $\HFa(Y)\otimes\FF_2[[\theta,\theta^{-1}]$ is independent of choices.
\end{remark}

\begin{remark}\label{rem:ordinary-covs}
  Theorem~\ref{thm:main} allows one to recover a result about ordinary
  double covers, by taking $L$ to be the unknot and choosing an
  interesting Seifert surface. Specifically, a double cover
  $\wt{Y}\to Y$ is induced by a $\ZZ$ cover if the corresponding
  element in $H^1(Y;\ZZ/2\ZZ)$ is the image of an element of
  $H^1(Y;\ZZ)$. In that case, the double cover is obtained by cutting
  $Y$ along a closed, orientable surface $F$ and gluing two copies of
  the result together. Let $F'$ be the complement of a small disk
  in $F$, and $U=\bdy F'$. It is not hard to see that the double cover branched along $U$
  with respect to the Seifert surface $F'$ is
  $\wt{Y}\# (S^2\times S^1)$. So, Theorem~\ref{thm:main} gives a
  spectral sequence relating
  $\HFa(\wt{Y}\# (S^2\times S^1))\cong \HFa(\wt{Y})\otimes H_*(S^1)$
  and $\HFa(Y)$. Such a spectral sequence was obtained by different
  techniques by Lipshitz-Treumann~\cite[Theorem 3]{LT:hoch-loc} (for torsion $\SpinC$-structures); this
  construction gives another explanation of the appearance of the
  $H_*(S^1)$ factor. This spectral sequence was also proved by
  Large~\cite[Theorem 1.4]{Large}, using his localization theorem; the
  argument we have just given essentially reduces to his.  (This
  remark was suggested to us by the referee.)
\end{remark}

\section{Applications}\label{sec:applications}

\begin{proof}[Proof of Corollary~\ref{cor:non-L}]
First, recall that, by Poincar\'e duality, a non-zero degree map $f\co N_1 \to N_2$ between closed, connected, oriented 3-manifolds induces an injection on cohomology with rational coefficients.  So, it follows from Lemma~\ref{lem:pullback-spinc-behaves} that if $\pi\co \Sigma(Y,L) \to Y$ is a branched double cover, then $\spinc \in \SpinC(Y)$ is torsion if and only if $\pi^*\spinc$ is.  Also, of course, $b_1(N_2) \leq b_1(N_1)$. 
 
Suppose that $b_1(\Sigma(Y,L)) = 0$, so $b_1(Y) = 0$ as well.  If $\HFred(\Sigma(Y,L)) = 0$, then Theorem~\ref{thm:main} implies that
\[
1 = \dim \HFa(\Sigma(Y,L), \pi^*\spinc) \geq \dim \HFa(Y,\spinc) \geq 1
\]
for all $\spinc \in \SpinC(Y)$, so $\HFred(Y) = 0$.  Hence, if $\Sigma(Y,L)$ is an L-space, so is $Y$.
  
Next, suppose that $b_1(\Sigma(Y,L)) = 1$.  If $N$ is a 3-manifold with $b_1(N) = 1$, then $\HFred(N) = 0$ if and only if $\HFa(N,\mathfrak{t}) = 0$ for non-torsion $\mathfrak{t}$ and $\dim \HFa(N,\mathfrak{t}) = 2$ for all torsion $\mathfrak{t}$.  (Recall that 2 is the lower bound for $\dim \HFa(N,\mathfrak{t})$ for torsion $\mathfrak{t}$, regardless of whether $\HFred$ is non-trivial.)  We now consider two cases: $b_1(Y) = 0$ or $b_1(Y) = 1$.  First, assume $b_1(Y) = 0$.  By Theorem~\ref{thm:main}, we see that $\dim \HFa(Y,\spinc) \leq 2$ for all $\spinc \in \SpinC(Y)$.  Since $\chi(\HFa(Y,\spinc)) = 1$, we must in fact have $\dim \HFa(Y,\spinc) = 1$ for all $\spinc$.  This is equivalent to $\HFred(Y) = 0$.  

Finally, assume $b_1(Y) = b_1(\Sigma(Y,L)) = 1$.  As in the previous case, Theorem~\ref{thm:main} guarantees 
\[
\dim \HFa(Y,\spinc) \leq \begin{cases} 2 \text{ if $\pi^*\spinc$ is torsion}\\ 0 \text{ if $\pi^*\spinc$ is non-torsion}.\end{cases}
\]
Since $\spinc$ is torsion if and only if $\pi^*\spinc$ is torsion, we have the desired constraints on $\HFa(Y)$ to guarantee that $\HFred(Y) = 0$.  
\end{proof}

\begin{proof}[Proof of Corollary~\ref{cor:involution}]
In the Seidel-Smith spectral sequence (see Section~\ref{sec:Large}),
the $E^1$ page is $\HFa(\Sigma(Y,K), \pi^* \spinc) \otimes
\FF_2[[\theta,\theta^{-1}]$, and the $d_1$ differential is given by
$(1+\tau_*) \theta$.  If $\tau_*$ was not the identity, the $d_1$
differential would not be identically 0, and we would deduce that
$\dim \HFa(Y, \spinc)$ is strictly less than $\dim \HFa(\Sigma(Y,K), \pi^* \spinc)$,
contradicting Theorem~\ref{thm:main}.   
\end{proof}

\begin{proof}[Proof of Proposition~\ref{prop:tangle-rep}]
Since $K$ has determinant 1, $\Sigma(Y,K)$ is a homology
sphere.  As $K$ is obtained by a rational tangle replacement,
$\Sigma(Y,K)$ is obtained by surgery on a knot $J$ in $\Sigma(Y,U) = Y
\# Y$.  Note that the surgery coefficient must be $1/n$ for some $n
\in \ZZ$ to produce a homology sphere.  Since $\Sigma(Y,K)$ is an
L-space, $Y$ is an L-space by Corollary~\ref{cor:non-L}, and so $Y \#
Y$ is an L-space.  In what follows, recall that if $Z$ is a homology sphere L-space and
a surgery $Z_\alpha(P)$ is an L-space then $|\alpha| \geq 2g(P) - 1$
(cf.~\cite[Proposition 9.6]{OS11:RatSurg}).

First, assume that $|n| \geq 2$, so by the previous remark $g(J) = 0$, i.e., $J$ is
unknotted in $Y \# Y$.  Therefore, $\Sigma(Y,K)$ is a homology sphere obtained by surgery along an unknot in $Y \# Y$, so $\Sigma(Y,K)$ is $Y \# Y$ as well.  By a result of Kim-Tollefson~\cite[Corollary 1]{KimTollefson80:pl}, because $Y$ is prime, the covering involution on $Y \# Y$ is either a connected sum of involutions on $Y$ or comes from taking the branched double cover of an unknot in an embedded $B^3$ in $Y$.  We must rule out the former.  In order for a connected sum of involutions on $Y$ to have a quotient to $Y$, we must be able to write $Y = \Sigma(Y,K')$ and $Y = \Sigma(S^3,K'')$; again, we are using the irreducibility of $Y$. If $Y = S^3$, then $K' = K'' = U$ and so $K$ is unknotted.  If $Y \neq S^3$, then $Y$ cannot admit a self-map of degree 2.  Indeed, if $Y$ is a prime L-space other than $S^3$, then $Y$ is the Poincar\'e homology sphere or is hyperbolic~\cite{Eftekhary-bordered, HRW}.  The case of the Poincar\'e homology sphere is handled by Boileau-Otal~\cite[Proposition 3.1]{BoileauOtal} and the hyperbolic case follows from supermultiplicativity of the Gromov norm, which is positive for hyperbolic manifolds, under non-zero degree maps. Thus, in this case, $K$ is unknotted.

Next, assume that $n = \pm 1$.  In this case, there are two options.  The first is that $J$ is unknotted, and by the previous argument, so is $K$.  The other is that $g(J) = 1$.  While a knot in $S^3$ with a non-trivial L-space surgery is fibered, a knot $P$ in a homology sphere L-space $Z$ with a non-trivial L-space surgery has the property that $P$ is fibered in some (not necessarily prime or proper) connected-summand of $Z$.  (The statement for knots in $S^3$ is due to Ghiggini \cite{Ghiggini08:FiberedGenusOne}.  The statement for knots in arbitrary homology spheres with irreducible exteriors follows from Ni's work \cite[Theorem 1.1 and Proof of Corollary 1.3]{Ni07:FiberedKnot}.)  Therefore, in our case, $J$ is a genus one fibered knot in a summand $Q$ of $Y \# Y$, which is necessarily a homology sphere L-space.  Of course, viewed as a knot in $Q$, $1/n$-surgery on $J$ is again an L-space homology sphere, since it is a summand of $\Sigma(Y,K)$.  By Baldwin's work~\cite{Baldwin:genusone}, the only homology sphere L-space, genus one fibered L-space knot pairs are $(S^3, \pm T_{2,3})$ and $\mp (\Sigma(2,3,5), F_5)$, where $F_5$ denotes the singular fiber of order 5, i.e. the core of $+1$-surgery on $T_{2,3}$.  (Here, the signs are chosen based on the sign of $n$.)  Note that in the former case, $1/n$-surgery produces $\pm \Sigma(2,3,5)$, while in the latter case, $1/n$-surgery produces $S^3$.   

In the first case, $J$ is a copy of $\pm T_{2,3}$ contained in an embedded 3-ball in $Y \# Y$, and so $\Sigma(Y,K) = Y \# Y \# \pm \Sigma(2,3,5)$. Since $Y$ is prime, it follows from Kim-Tollefson~\cite[Corollary 1]{KimTollefson80:pl} that $K$ must be a knot in an embedded 3-ball in $Y$ with branched double cover $\pm \Sigma(2,3,5)$.  (Here we are using that $\Sigma(2,3,5)$ is not a branched or unbranched double cover of itself, which follows from \cite[Proposition 3.1]{BoileauOtal}.)  By a result of Watson~\cite[Theorem 6.2]{WatsonSurgery} $K$ is a copy of $\mp T_{3,5}$ in an embedded $B^3$ in $Y$.  In the second case, we see that the Poincar\'e homology sphere is a summand of $Y \# Y$ and hence of $Y$.  Because we assumed $Y$ is irreducible, $Y$ is the Poincar\'e homology sphere, and $\Sigma(Y,K)$ is one copy of the Poincar\'e homology sphere.  Since there is no knot in the Poincar\'e homology sphere whose branched double cover is again the Poincar\'e homology sphere, this last case does not arise.      
\end{proof}

\begin{remark}
If $Y$ is not prime, similar characterizations can likely be obtained, but it requires a more tedious analysis of the possible involutions on the relevant 3-manifolds.\end{remark}

\begin{remark}
Assuming the Heegaard Floer Poincar\'e conjecture, this proposition can be proved without requiring the results from this paper, since the involutions on $S^3$ and connected sums of the Poincar\'e homology sphere are well understood.  
\end{remark}

\subsection{Analogue in sutured Floer homology}\label{sec:sutured}
In this section we prove an analogue of Theorem~\ref{thm:main} for sutured Floer homology.  
Let $(M,\gamma)$ be a balanced sutured manifold and $L \subset M$ a
nullhomologous link in the interior of $M$.  Then, there is a natural sutured structure
$\wt{\gamma}$ on $\partial \Sigma(M,L)$: the sutures are the preimage
of the sutures of $M$ under the covering map $\pi\co \bdy
\Sigma(M,L)\to\bdy M$, and the positive / negative regions
$\wt{R}_\pm$ are the preimages of the positive / negative regions in
$\bdy M$.  Since
$\chi(\wt{R}_+)=2\chi(R_+)=2\chi(R_-)=\chi(\wt{R}_-)$, $(\Sigma(M,L),\wt{\gamma})$ is also balanced.

\begin{proof}[Proof of Proposition~\ref{prop:sutured-inequality}]
For simplicity, we assume that $K$ is a knot.  The extension from knots to links is analogous to the closed case.  

By a \emph{doubly-pointed sutured Heegaard diagram} for $(M,\gamma,K)$
we mean a sutured
Heegaard diagram $(\Sigma,\alphas,\betas)$ for $(M,\gamma)$ together with a pair of
points $z,w\in \Sigma\setminus(\alphas\cup\betas)$ so that
$(\Sigma\setminus\nbd(\{z,w\}),\alphas,\betas)$ is a sutured Heegaard
diagram for $M\setminus\nbd(K)$, with two meridional sutures around $K$.
Call
$(\Sigma,\alphas,\betas,z,w)$ \emph{admissible} if the sutured
Heegaard diagram $(\Sigma\setminus\nbd(z),\alphas,\betas)$ is
admissible. 

A simple Morse-theory argument shows that every knot in the interior
of $M$ is represented by some doubly-pointed Heegaard diagram
(compare~\cite[Proposition 2.3]{Juhasz06:Sutured}). Further, any
doubly-pointed Heegaard diagram can be made weakly admissible by an
isotopy of the $\alpha$-circles (cf.~\cite[Proposition
3.15]{Juhasz06:Sutured}). 

So, choose an admissible doubly-pointed sutured Heegaard diagram
$\HD=(\Sigma, \alphas, \betas, z, w)$ for $K \subset (M,\gamma)$. A
Seifert surface for $K$ transverse to $\Sigma$ induces a branched
double cover $\wt{\Sigma}$ of $\Sigma$, branched over $\{z,w\}$. If we
let $\wt{\alphas}$, $\wt{\betas}$, $\wt{z}$, and $\wt{w}$ be the
preimages of $\alphas$, $\betas$, $z$, and $w$ under the branched
covering map then
$(\wt{\Sigma}, \wt{\alphas}, \wt{\betas}, \wt{z}, \wt{w})$ is a
doubly-pointed sutured Heegaard diagram representing
$\wt{K}=\pi^{-1}(K)$ in $(\Sigma(M,K), \wt{\gamma})$.
(This is clear, for example, by considering a Morse-theoretic
interpretation of sutured Heegaard diagrams.)

Let $d$ be the number of $\alpha$-circles in the Heegaard diagram
$\HD$. 
We apply Large's theorem to prove Lemma~\ref{lem:sutured-Large} below,
which yields a spectral sequence with $E^1$ page the Floer homology
$\HF(\TT_{\wt{\alpha}}, \TT_{\wt{\beta}}) \otimes \FF_2[[\theta, \theta^{-1}]$ computed in $\Sym^{2d}(\wt{\Sigma} \setminus
\{\wt{z}\})$ and $E^\infty$ page the Floer homology $\HF(\TT_\alpha,
\TT_\beta) \otimes \FF_2[[\theta, \theta^{-1}]$ computed in
$\Sym^{d}(\Sigma \setminus \{z\})$.

Note that these Lagrangian Floer homologies are not describing the sutured Floer homologies in the proposition.  Given a balanced sutured manifold $(Z,\eta)$, define $(Z^\circ, \eta^\circ)$ to be the balanced sutured manifold obtained by removing an embedded 3-ball from $M$ and adding an equatorial suture on the additional 2-sphere component in the boundary.  So,
\begin{align*}
HF(\TT_{\wt{\alpha}}, \TT_{\wt{\beta}}) &\cong \SFH(\Sigma(M,K)^\circ, \wt{\gamma}^\circ)\\
HF(\TT_\alpha, \TT_\beta) &\cong \SFH(M^\circ,  \gamma^\circ).
\end{align*}
The K\"unneth theorem for sutured Floer homology~\cite[Proposition 9.15]{Juhasz06:Sutured} implies that
\begin{equation}\label{eq:sutured-stabilization}
\SFH(Z^\circ, \eta^\circ) \cong \SFH(Z,\eta) \otimes H_*(S^1),
\end{equation}
which gives the desired result.
\end{proof}

\begin{lemma}\label{lem:sutured-Large} 
Consider an admissible doubly-pointed sutured Heegaard diagram
$(\Sigma, \alphas, \betas, z, w)$ for $K \subset (M,\gamma)$, and
let $(\wt{\Sigma}, \wt{\alphas}, \wt{\betas}, \wt{z}, \wt{w})$ be the
associated diagram for $(\Sigma(M,K), \gamma)$.  Then, there is a
spectral sequence with $E^1$-page the Floer homology
$\HF(\TT_{\wt{\alpha}}, \TT_{\wt{\beta}}) \otimes
\FF_2[[\theta,\theta^{-1}]$ inside $\Sym^{2d}(\wt{\Sigma} \setminus \{\wt{z}\})$ and
$E^\infty$-page the Floer homology $\HF(\TT_\alpha, \TT_\beta)\otimes
\FF_2[[\theta, \theta^{-1}]$ inside $\Sym^{d}({\Sigma} \setminus \{ {z}\})$.  
\end{lemma}
\begin{proof}
The proof that the symplectic hypotheses of Theorem~\ref{thm:Large}
are satisfied is similar to the proof of
Proposition~\ref{prop:symplectic-hypotheses}, and is left to the
reader. It remains to show that there is a tangent-normal isomorphism 
\[
(T \Sym^d(\Sigma \setminus \{z\}) , T \TT_{\alpha}, T\TT_{\beta}) \cong (N \Sym^d(\Sigma \setminus \{z\}), N\TT_{\alpha}, N \TT_{\beta}).  
\] 
The argument proceeds in two steps as in the closed case.  First,
Large's argument~\cite[Proof of Propositions 10.1 and 10.2]{Large} establishes an isomorphism
\[
\Phi_1\co (T \Sym^d(\Sigma \setminus \{w,z\}) , T \TT_{\alpha}, T\TT_{\beta}) \cong (N \Sym^d(\Sigma \setminus \{w,z\}), N\TT_{\alpha}, N \TT_{\beta}).
\]
Since $z$ and $w$ lie in the same component of $\Sigma$, as a special
case we again get an isomorphism
\[
  \Phi_2\co  T \Sym^d(\Sigma \setminus \{z\})\cong N \Sym^d(\Sigma \setminus \{z\})
\]
which may not respect the tangent and normal bundles to
the tori (cf. Lemma~\ref{lemma:bigiso}).

We show that the first of these isomorphisms can be modified to extend
this over $\{w\}\times\Sym^{d-1}(\Sigma\setminus\{z\})$.  As in
Section~\ref{sec:normalIso}, the space $\Sigma \setminus \{z\}$
deformation retracts onto a wedge of circles, so by
Lemma~\ref{lem:sym-is-torus} any $g$-fold symmetric product
$\Sym^d(\Sigma \setminus \{z\})$ has the homotopy type of a skeleton
of a torus. It follows by the same argument as
Corollary~\ref{corollary:integral-iso} that the Chern character is an
integral isomorphism
$\chern\co \Iso(E,E)\to H^{\mathrm{odd}}(\Sym^d(\Sigma\setminus
\{z\}))$ for any complex vector bundle $E$ over
$\Sym^d(\Sigma\setminus \{z\})$, and similarly for
$\Sym^d(\Sigma\setminus
\{z,w\})$. Lemma~\ref{lemma:works-for-cohomology} still holds in this
context, with the same proof. So, the proof of
Proposition~\ref{prop:trivialization} applies to show that $\Phi_2 \circ \Phi_4$ gives a 
tangent-normal isomorphism, as desired.
\end{proof}

\begin{corollary}
Let $(M,\gamma)$ be a balanced sutured manifold and $L \subset M$ a nullhomologous link.  If $(M,\gamma)$ is a taut sutured manifold and $\Sigma(M,L)$ is irreducible, then $(\Sigma(M,L), \wt{\gamma})$ is taut as well.
\end{corollary}
\begin{proof}
An irreducible balanced sutured manifold has non-vanishing sutured Floer homology if and only if it is taut \cite[Proposition 9.18]{Juhasz06:Sutured}, \cite[Theorem 1.4]{Juhasz08:SuturedDecomp}.  The theorem therefore follows from Proposition~\ref{prop:sutured-inequality}.
\end{proof}

\begin{corollary}
Let $Y$ be a closed, connected, oriented 3-manifold, $L \subset Y$ a link, and
$Q \subset Y\setminus L$ a link which is nullhomologous in $Y\setminus
L$.  Let $\wt{L}$ be the preimage of $L$ inside $\Sigma(Y,Q)$. Then, there is a rank inequality 
\[
\dim \HFLa(\Sigma(Y,Q), \widetilde{L}) \geq \dim \HFLa(Y,L).  
\]
\end{corollary}
Here, if $L$ is not nullhomologous, by $\HFLa(Y,L)$ we mean the
sutured Floer homology of $Y\setminus\nbd(L)$ with meridional sutures. For a more concrete case, if $Y = S^3$ and $Q$ is an unknot, then this gives a rank inequality for the knot Floer homology of certain 2-periodic links, which was proved by the first author.  In this case, the condition that $Q$ be nullhomologous in the exterior of $L$ is equivalent to the quotient link having linking number 0 with the axis of symmetry.
\begin{proof}
Let $M$ denote the exterior of $L$ and $\gamma$ consist of a pair of
meridional sutures for each toral boundary component, so
$\SFH(M,\gamma) \cong \widehat{HFL}(Y,L)$.  Similarly,
$\SFH(\Sigma(M,Q), \wt{\gamma}) \cong \widehat{HFL}(\Sigma(Y,Q),
\wt{L})$.  Thus, the result follows from Proposition~\ref{prop:sutured-inequality}, since $Q$ is nullhomologous in $M$ by assumption.  
\end{proof}

\begin{remark}
Perhaps one could use Proposition~\ref{prop:sutured-inequality} to recover classical theorems in equivariant 3-manifold topology for involutions (with suitable constraints on the branch set), such as the equivariant Dehn's lemma~\cite{MeeksYau81:loop,Edmonds86:Dehn}.
\end{remark}

\bibliographystyle{hamsalpha}
\bibliography{heegaardfloer}
\end{document}


%% file: defs.tex
\newcommand{\RR}{\mathbb R}
\newcommand{\CC}{\mathbb C}
\newcommand{\ZZ}{\mathbb Z}
\newcommand{\QQ}{\mathbb Q}

\newcommand{\FF}{\mathbb F}

\newcommand{\co}{\nobreak\mskip2mu\mathpunct{}\nonscript
  \mkern-\thinmuskip{:}\penalty300\mskip6muplus1mu\relax}

\newcommand{\bdy}{\partial}
\newcommand{\into}{\hookrightarrow}

\newcommand{\lbracket}{[}
\newcommand{\rbracket}{]}

\newcommand{\spinc}{\mathfrak s}
\newcommand{\spinct}{\mathfrak t}

\DeclareMathOperator{\Sym}{Sym}

\DeclareMathOperator{\Iso}{Iso}
\DeclareMathOperator{\Map}{Map}

\DeclareMathOperator{\spin}{spin}

\newcommand{\SpinC}{\spin^c}

\renewcommand{\emptyset}{\varnothing}

\theoremstyle{plain}

\numberwithin{equation}{section}
\newtheorem{theorem}[equation]{Theorem}

\newtheorem{proposition}[equation]{Proposition}
\newtheorem{lemma}[equation]{Lemma}
\newtheorem{corollary}[equation]{Corollary}

\newtheorem{definition}[equation]{Definition}

\theoremstyle{definition}

\theoremstyle{remark}

\newtheorem{remark}[equation]{Remark}

\newcommand{\HF}{\mathit{HF}}

\newcommand{\HFa}{\widehat{\mathit{HF}}}

\newcommand{\SFH}{\mathit{SFH}}

\newcommand{\CF}{{\mathit{CF}}}

\newcommand{\HFKa}{\widehat{\mathit{HFK}}}
\newcommand{\HFLa}{\widehat{\mathit{HFL}}}

\newcommand{\HFred}{\mathit{HF}_{\!\!\mathrm{red}}}
\newcommand{\x}{\mathbf x}
\newcommand{\y}{\mathbf y}

\newcommand\HH{\mathit{HH}}

\newcommand\Hochschild\HH

\newcommand{\alphas}{{\boldsymbol{\alpha}}}
\newcommand{\betas}{{\boldsymbol{\beta}}}

\newcommand\Id{\mathbb{I}}

\newcommand{\Field}{\FF_2}

\DeclareMathOperator{\nbd}{nbd}

\newcommand{\dbar}{\bar{\partial}}
\newcommand{\Heegaard}{\mathcal{H}}
\newcommand{\HD}{\Heegaard}

\newcommand{\wt}[1]{\widetilde{#1}{}}

\makeatletter
\newcommand\honestalg[3]{\bigl\lbracket
\begin{smallmatrix} #1\@ifempty{#3}{}{&#3} \\ #2 \end{smallmatrix}
\bigr\rbracket}

\makeatother
\newcommand{\lab}[1]{$\scriptstyle #1$}

\newcommand{\pt}{\mathit{pt}}

\newcommand{\fix}{\mathit{fix}}

\newcommand{\onto}{\twoheadrightarrow}
\newcommand{\ul}{\underline}

\makeatletter

\newcommand*\wthelper[2]{%
        \hbox{\dimen@\accentfontxheight#1%
                \accentfontxheight#11.3\dimen@
                $\m@th#1\widetilde{#2}$%
                \accentfontxheight#1\dimen@
        }%
}
\newcommand*\accentfontxheight[1]{%
        \fontdimen5\ifx#1\displaystyle
                \textfont
        \else\ifx#1\textstyle
                \textfont
        \else\ifx#1\scriptstyle
                \scriptfont
        \else
                \scriptscriptfont
        \fi\fi\fi3
}
\makeatother

\makeatletter
\newlength\xvec@height%
\newlength\xvec@depth%
\newlength\xvec@width%
\newcommand{\xvec}[2][]{%
  \ifmmode%
    \settoheight{\xvec@height}{$#2$}%
    \settodepth{\xvec@depth}{$#2$}%
    \settowidth{\xvec@width}{$#2$}%
  \else%
    \settoheight{\xvec@height}{#2}%
    \settodepth{\xvec@depth}{#2}%
    \settowidth{\xvec@width}{#2}%
  \fi%
  \def\xvec@arg{#1}%
  \def\xvec@dd{:}%
  \def\xvec@d{.}%
  \raisebox{.2ex}{\raisebox{\xvec@height}{\rlap{%
    \kern.05em
    \begin{tikzpicture}[scale=1]
    \pgfsetroundcap
    \draw (.05em,0)--(\xvec@width-.05em,0);
    \draw (\xvec@width-.05em,0)--(\xvec@width-.15em, .075em);
    \draw (\xvec@width-.05em,0)--(\xvec@width-.15em,-.075em);
    \ifx\xvec@arg\xvec@d%
      \fill(\xvec@width*.45,.5ex) circle (.5pt);%
    \else\ifx\xvec@arg\xvec@dd%
      \fill(\xvec@width*.30,.5ex) circle (.5pt);%
      \fill(\xvec@width*.65,.5ex) circle (.5pt);%
    \fi\fi%
    \end{tikzpicture}%
  }}}%
  #2%
}
\makeatother

\DeclareMathOperator{\PD}{PD}

\newcommand{\TT}{\mathbb{T}}

\newcommand{\lra}{\longrightarrow}

\DeclareMathOperator{\chern}{ch}
\DeclareMathOperator{\Cyl}{Cyl}
\DeclareMathOperator{\Span}{Span}

\newcommand{\twHF}{\HF_{\!\mathit{tw}}}
